\documentclass[11pt, a4paper]{amsart}

\usepackage{mystyles}

\begin{document}

\title[Odd Fredholm modules over dg algebras]{Odd Fredholm modules over
	differential graded algebras: Chern character and spectral flow}
\author{Jonas Miehe}
\date{November 19, 2026}

\subjclass[2020]{Primary 58B34; Secondary 19D55, 46L87, 58J20, 58J30}
\keywords{Entire cyclic cohomology, JLO cocycle, $\vartheta$-summable
	Fredholm modules, differential graded algebras, Chen normalization, odd
	Chern character, spectral flow, loop spaces}

\begin{abstract}
	Let $\Omega$ be a locally convex \emph{differential graded} algebra.
	For $q \in \mathbb{N}_0$ we construct the Chern character of
	$\vartheta$-summable $\mathcal{C}_q$-Fredholm modules over $\Omega$, an
	entire cyclic cocycle which extends the JLO cocycle to the differential
	graded setting.
	For $q = 0$ it is the Chern character of even $\vartheta$-summable Fredholm
	modules over $\Omega$ due to Güneysu and Ludewig; for $q = 1$ it yields a
	Chern character for odd $\vartheta$-summable Fredholm modules over
	$\Omega$, which reduces to Getzler's odd Chern character for $\Omega$
	trivially graded with trivial differential.
	We prove that this cocycle is Chen normalized and invariant under
	homotopies.
	Moreover, we generalize the constructions by Cacciatori and Güneysu to
	arbitrary $\Omega$ and $g \in \mathrm{GL}_m(\Omega^0)$ to get an odd chain
	$\mathrm{Ch}(g)$ and prove the following odd index theorem:
	the pairing of the odd Chern character with $\mathrm{Ch}(g)$ may be written
	as an explicit integral formula which equals $\sqrt{\pi}$ times the
	spectral flow from $Q$ to $\mathbf{c}(g)^{-1} Q \mathbf{c}(g)$ whenever
	$\mathbf{c}(g)$ is unitary.
	For odd-dimensional spin manifolds (where $\mathbf{c}$ is Clifford
	multiplication) this supplies the cochain side of the loop space approach
	to the odd-index theorem.
\end{abstract}
\maketitle

\section{Introduction}
Over the last decades, deep connections between index theory, noncommutative
geometry and cyclic homology have been established using the
algebraic framework provided by Fredholm modules.
Given an even $\vartheta$-summable Fredholm module over a unital Banach
$*$-algebra, $\mathcal{A}$, Connes \cite{ConnesCyclicCohomology} and
Jaffe, Lesniewski and Osterwalder \cite{JLO} constructed a Chern
character, an entire cyclic
cocycle on $\mathcal{A}$, which pairs with
$K_0(\mathcal{A})$ and thus produces a noncommutative index theorem
\cite{GetzlerSzenes}, generalizing previous results in the
finitely-summable case.
The weaker $\vartheta$-summability assumption together with the admission of
infinite chains subject to entire growth conditions makes it possible to treat
infinite-dimensional situations, such as those arising in quantum field theory.

This program was carried over to the odd case by Getzler \cite{GetzlerOdd}, who
introduced the Chern character of an \emph{odd} $\vartheta$-summable Fredholm
module -- the \textit{odd JLO cocycle} -- which pairs with odd $K$-theory
$K_1(\mathcal{A})$.
In the resulting odd index theorem the index is replaced by the spectral flow.

Motivated by an infinite-dimensional variant of Duistermaat-Heckman
localization in equivariant loop space cohomology, which yields a
remarkably short -- however formal --
proof of the Atiyah-Singer index theorem, Güneysu and Ludewig
\cite{GüneysuLudewig} extended the theory of even $\vartheta$-summable Fredholm
modules to the setting of locally convex \emph{differential graded} algebras.
They constructed an analytic cyclic cocycle which satisfies a
noncommutative index theorem, which is invariant under homotopies of
Fredholm modules, and which coincides with the JLO cocycle if the algebra
is trivially graded with trivial differential.
The differential graded setting is what makes these constructions geometric: it
allows one to formulate and prove a localization formula on the loop space in a
mathematically rigorous way through the use of Chen's iterated integrals,
confirming conjectures of Atiyah, Bismut and Witten
\cite{AtiyahCircular,BismutCMP85}.

This circle of ideas, which goes back at
least to Getzler, Jones and Petrack \cite{GetzlerJonesPetrack}, remains active,
see for instance \cite{GüneysuMiehe} for a stochastic refinement of the
Duistermaat-Heckman localization formula on loop spaces based on iterated
operator integrals that also underlie the present paper.

\subsection*{The odd case}
In the differential graded category the theory for odd Fredholm modules has
not been developed so far.
On the side of chains, Cacciatori and Güneysu
\cite{GüneysuCacciatori} constructed, for a smooth map $g$ from a compact
manifold into a unitary group, an odd Chern character $\mathrm{Ch}(g)$ in a
Chen normalized entire cyclic complex, which is mapped by Chen's iterated
integrals to the odd Bismut--Chern character on the loop space.
As is done in \cite{GüneysuCacciatori}, this may be seen as the chain-level
ingredient of a loop space derivation of the odd index formula expressing the
spectral flow from a Dirac operator $D$ to $g^{-1}Dg$ as $\int
\hat{A}\wedge\mathrm{ch}^-(g)$.
What has been missing is the corresponding cochain-level object: a Chern
character of odd $\vartheta$-summable Fredholm modules over a differential
graded algebra compatible with Chen normalization and homotopies,
together with an index theorem computing its pairing with $\mathrm{Ch}(g)$.
The purpose of this paper is to supply these.

Treating the even and the odd case simultaneously is achieved by an action of
a complex Clifford algebra as is customary since Kasparov's work \cite{Kasparov} and
Quillen's treatment of odd superconnections \cite{Quillen}.
For $q\in\N_0$ we introduce \emph{$\vartheta$-summable
$\mathcal{C}_q$-Fredholm modules} over a locally convex differential graded
algebra $\Omega$ (\cref{Def FM}) and define their Chern character by means of
a Clifford supertrace $\CStr$ (\cref{Clifford supertrace}).
An even $\vartheta$-summable Fredholm module over $\Omega$ in the sense of
\cite{GüneysuLudewig} is the same as a $\mathcal{C}_0$-Fredholm module, and an
odd $\vartheta$-summable Fredholm module $\mathscr{M}$ over $\Omega$
canonically induces a $\vartheta$-summable $\mathcal{C}_1$-Fredholm module
$\tilde{\mathscr{M}}$ (\cref{Section: odd FM}).
Our first main result collects the properties of the resulting Chern character.

\begin{thmALPH}[\cref{Thm ChM}]\label{Thm A}
	Let $\mathscr{M}$ be a $\vartheta$-summable $\mathcal{C}_q$-Fredholm module
	over a locally convex differential graded algebra $\Omega$.
	There exists a canonically given analytic cochain $\Chern$ on the cyclic
	complex over $\Omega$, the \emph{Chern character of
	$\mathscr{M}$}, with the following properties:
	\begin{enumerate}
		\item $\Chern$ is coclosed.
		\item If $\mathscr{M}_\T$ denotes the acyclic extension of
		$\mathscr{M}$ (a weak module over the acyclic extension $\Omega_\T$ of
		$\Omega$), then $\ChernT$ is Chen normalized, i. e. it vanishes on the
		subcomplex defining the Chen normalization of the
		entire cyclic complex over $\Omega_\T$.
		\item If $\mathscr{M}^s$, $s \in [0,1]$, is a homotopy of
		$\vartheta$-summable
		$\mathcal{C}_q$-Fredholm modules over $\Omega$, then
		$\mathrm{Ch}_{\mathscr{M}^0_\T}$ and
		$\mathrm{Ch}_{\mathscr{M}^1_\T}$ are cohomologous in the complex of
		Chen normalized analytic cochains.
	\end{enumerate}
\end{thmALPH}
For $q=0$ the cochain $\Chern$ coincides with the Chern character of
\cite{GüneysuLudewig}.
For $q=1$, the construction applies to the
$\mathcal{C}_1$-modules induced by odd (weak) modules, and we show in
\cref{Lemma Getzler} that for an algebra with trivial grading and trivial
differential it reduces to Getzler's odd JLO cocycle; thus
\cref{Thm A} provides a common differential graded extension of the even
theory of \cite{GüneysuLudewig} and the odd theory of \cite{GetzlerOdd}.

Next, we extend the chains of \cite{GüneysuCacciatori} to the present
setting.
\begin{thmALPH}[\cref{Thm Chg}]\label{Thm B}
	Let $\Omega$ be a locally convex differential graded algebra.
	For every $g \in \mathrm{GL}_m(\Omega^0)$ there exists a canonically given
	odd entire cyclic chain $\mathrm{Ch}(g)$ over $\Omega_\T$.
	This chain is closed in the Chen normalized entire cyclic complex, and its
	homology class is invariant under smooth homotopies of $g$.
\end{thmALPH}
Using the combinatorial machinery of the Chern character, we prove a
noncommutative index theorem by pairing the Chern character with the entire
chain from above:
\begin{thmALPH}[\cref{Thm index}]\label{Thm C}
	Let $\mathscr{M} = (\HH, \mathbf{c}, Q)$ be an odd $\vartheta$-summable
	Fredholm
	module over a locally convex differential graded algebra $\Omega$, let
	$\tilde{\mathscr{M}}$ be the associated
	$\mathcal{C}_1$-Fredholm module and let $g\in\mathrm{GL}_m(\Omega^0)$.
	Then
	\begin{equation*}\label{formel1}
		\langle \mathrm{Ch}_{\tilde{\mathscr{M}}_{\T}}
		,\mathrm{Ch}(g)\rangle =\int_0^1 \Tr(\dot{Q}_s e^{-Q_s^2}) \,
		\dint s
	\end{equation*}
	where $Q_s = (1-s) Q_m + s \mathbf{c}(g)^{-1} Q_m \mathbf{c}(g)$ and
	$\mathbf{c}(g) = (\mathbf{c}(g_i^j))_{i,j}$ act on
	$\HH^m$ and $\dot{Q}_s= \frac{d}{ds}Q_s$.
	If $\mathbf{c}(g)$ is unitary, the right-hand side equals
	$\sqrt{\pi}\, \mathrm{sf}(Q_m, \mathbf{c}(g)^{-1} Q_m
	 \mathbf{c}(g))$.
\end{thmALPH}

\cref{Thm C} constitutes the differential graded counterpart of Getzler's
spectral flow integral formula \cite[Theorem 4.3]{GetzlerOdd}, recalled in
\cref{Getzler sf} below, and the odd counterpart of the index theorem
\cite[Theorem C]{GüneysuLudewig}.
Its proof is based on matching the
perturbation series of $e^{-Q_s^2}$ with the quantization map of the
$\mathcal{C}_1$-module evaluated on the building blocks of $\mathrm{Ch}(g)$
(\cref{Lemma perturbation quantization}).

In \cref{Section: example} we discuss the basic geometric example: given a
compact odd-dimensional spin manifold $X$, the Dirac operator induces an odd
$\vartheta$-summable Fredholm module over the exterior algebra $\Omega(X)$, and
\cref{Thm C} expresses the spectral flow from $D$ to $g^{-1}Dg$, for smooth
$g\colon X\to\mathrm{U}(m)$, as the pairing of its Chern character with
$\mathrm{Ch}(g)$.
Since $\mathrm{Ch}(g)$ corresponds under Chen's iterated
integrals to the odd Bismut--Chern character on the loop space
\cite{GüneysuCacciatori}, this provides the analytic side of a loop space
approach to the odd index theorem, in analogy with \cite{GüneysuLudewig} in
the even case.

\subsection*{Organization}
\cref{Section: Algebraic preliminaries} recalls the cyclic and bar
complexes, analytic cochains and entire chains, and the Chen normalized
complex, following \cite{GüneysuLudewig}; readers familiar with
\cite{GüneysuLudewig} may consult it for notation only.
\cref{Section: Cq FM} introduces $\vartheta$-summable
$\mathcal{C}_q$-Fredholm modules and their homotopies, and
\cref{Section: odd FM} explains how odd modules fit into this framework.
\cref{Section: Chern character} contains the construction of $\Chern$ and the
proof of \cref{Thm A}; the main new point is the cyclicity of the Clifford
supertrace (\cref{Clifford trace cyclic sum property}), and
\cref{Section: Getzler} compares the construction with Getzler's odd Chern
character. \cref{Section: index} is devoted to the odd index theorem: we
construct $\mathrm{Ch}(g)$ and prove \cref{Thm B}
(\cref{Section: definition cherng}), set up the spectral flow side
(\cref{Section: spectral flow}), and prove \cref{Thm C}
(\cref{Section: spectral flow pairing 1,Section: spectral flow pairing 2}).
\cref{Section: example} treats the example of spin manifolds.

\subsection*{Acknowledgements}
This paper is based on part of the author's PhD thesis \cite{MieheThesis}, written under the
supervision of Batu Güneysu at Chemnitz University of Technology.
The author thanks Sergio Cacciatori and Shu Shen for helpful discussions.

\section{Algebraic preliminaries}\label{Section: Algebraic preliminaries}
This section recalls the algebraic framework of \cite{GüneysuLudewig} and
fixes the sign conventions used throughout; apart from notation, nothing in
it is new.
All vector spaces and algebras will be assumed to be over $\C$, and we always
work with \textit{graded} vector spaces unless specified otherwise.
On a $\Z$-graded vector space we will use the canonical $\Z_2$-grading
determined by the even and odd elements and correspondingly, tensor products
and direct sums of $\Z$-graded vector spaces are canonically $\Z$-graded and
thereby $\Z_2$-graded.
If a vector space carries no specified grading, we turn it into a
$\Z$-graded (and therefore $\Z_2$-graded) vector space by letting all
elements have degree zero.

Given $\Z_2$-graded vector spaces $V$ and $W$, the vector space
$\Hom(V, W)$ is $\Z_2$-graded as usual.
Any $A \in \End(V)$ gives rise to a $\Z_2$-graded dual map $A^{\vee} \colon
\Hom(V, W) \to \Hom(V, W)$ with values in $W$ given by
\begin{align*}
	(A^{\vee}\ell) (v) \coloneqq (-1)^{|A||\ell|}\ell (A(v)), \quad \ell\in
	\mathrm{Hom}(V,W), \quad v\in V.
\end{align*}
We use this definition of the dual map instead of the usual one unless
specified otherwise.
For $A, B\in \End(V)$ we will denote by $[A, B] \in \End(V)$
the graded commutator (or supercommutator), i. e.
\begin{equation*}
	[A,B] \coloneqq AB-(-1)^{|A||B|}BA,
\end{equation*}
and we will use the notation $[[A, B]] \coloneqq AB + BA$ for the
anticommutator.
\subsection{Differential graded algebras}
\begin{defn}
	A \textit{differential graded algebra} (\enquote{dg algebra}) $\Omega$ is
	an associative, $\Z$-graded algebra with unit $\mathbf{1}$ equipped
	with a degree $+1$ differential $d$ subject to the graded Leibniz
	rule
	\begin{equation*}
		d(ab) = (da)b + (-1)^{\lvert a \rvert}adb
	\end{equation*}
	for  $a, b \in \Omega$ of pure degree.

	If $\Omega$ carries a locally convex topology\footnote{We always assume
		locally convex topologies to be Hausdorff.} such that $d$ is
	continuous, multiplication is jointly continuous and $\Omega$ is the
	topological direct
	sum of its homogeneous subspaces, then $\Omega$ is a \textit{locally
		convex dg algebra}.
\end{defn}
\begin{eg}
	Let $X$ be a smooth manifold and $\Omega \coloneqq \Omega(X)$ be the
	space of differential forms on $X$.
	Together with the exterior differential and the wedge product, $\Omega$
	is a dg algebra which carries a canonical Fréchet topology turning it
	into a locally convex dg algebra.
\end{eg}
The following construction will be important later on:
\begin{eg}\label{Acyclic extension}
	Let $\Omega$ be a locally convex dg algebra.
	Denote by
	\begin{equation*}
		\Omega_{\mathbb{T}} \coloneqq \Omega[\sigma]
	\end{equation*}
	the $\Z$-graded unital algebra obtained by adjoining the formal variable
	$\sigma$ of degree $-1$ with $\sigma^2 = 0$ and supercommuting with
	elements of $\Omega$.
	By definition, $\theta \in \Omega_{\mathbb{T}}$ can be uniquely written as
	$\theta = \theta' + \sigma \theta''$ where $\theta', \theta'' \in \Omega$.
	We define a differential $d_{\mathbb{T}}$ via
	\begin{equation*}
		d_{\mathbb{T}} \coloneqq d - \iota \quad \text{where} \quad d \theta
		\coloneqq
		d \theta' - \sigma d \theta'' \quad \text{and} \quad \iota \theta
		\coloneqq
		\theta''.
	\end{equation*}
	Using the isomorphism of graded vector spaces $\Omega_{\mathbb{T}} \simeq
	\Omega \oplus \Omega[1]$ (where $\Omega[1]$ is equal to $\Omega$ with
	degrees shifted upwards by one), it follows that $\Omega_{\mathbb{T}}$
	becomes a locally convex dg algebra which we will call the
	\textit{acyclic extension of $\Omega$}.
\end{eg}
Throughout the rest of this section, let $\Omega$ be a locally convex dg
algebra.
\subsection{The (reduced) cyclic complex}\label{Section: cyclic complex}
\textit{The (reduced) cyclic complex associated to $\Omega$} is the vector
space
\begin{equation*}
	\mathrm{C}(\Omega) \coloneqq \bigoplus_{N=0}^{\infty}\Omega \otimes
	\underline{\Omega}[1]^{\otimes N}
\end{equation*}
with grading determined by
\begin{equation*}
	\mathrm{C}_n(\Omega) = \bigoplus_{N=0}^{\infty} \bigoplus_{l_0 + \ldots +
		l_N = n + N} \Omega^{l_0} \otimes \underline{\Omega}^{l_1} \otimes
	\ldots
	\otimes \underline{\Omega}^{l_N}
\end{equation*}
for $n \in \Z$ where $\underline{\Omega} \coloneqq \Omega / \C \mathbf{1}$ is
the quotient vector space.
We will usually denote elements of $\mathrm{C}(\Omega)$ by $(\theta_0,
\ldots, \theta_N)$ instead of $\theta_0 \otimes \ldots \otimes \theta_N$.
We define two degree $+1$ differentials on $\mathrm{C}(\Omega)$ via
\begin{align}
	\underline{d}(\theta_0, \ldots, \theta_N) & \coloneqq (d
	\theta_0, \theta_1, \ldots, \theta_N) - \sum_{k=1}^N
	(-1)^{m_{k-1}}(\theta_0, \ldots, d
	\theta_k, \ldots, \theta_N) \notag \\
	\begin{split}\label{Definition bunderline}
		\underline{b}(\theta_0, \ldots, \theta_N) & \coloneqq
		-\sum_{k=0}^{N-1}(-1)^{m_{k}}(\theta_0, \ldots, \theta_k
		\theta_{k+1},
		\ldots, \theta_N) \\
		& \qquad + (-1)^{(\lvert \theta_{N} \rvert - 1)
			m_{N-1}}(\theta_N\theta_0, \ldots,\theta_{N-1})
	\end{split}
\end{align}
where $m_k \coloneqq \lvert \theta_0 \rvert + \ldots + \lvert \theta_k \rvert
- k$.

The \textit{Connes operator on $\mathrm{C}(\Omega)$} is the degree $-1$
differential defined via
\begin{equation*}
	\underline{B}(\theta_0, \ldots, \theta_N) \coloneqq \sum_{k=0}^N
	(-1)^{(m_{k-1} + 1)(m_N - m_{k-1})}(\mathbf{1}, \theta_{k},
	\ldots,
	\theta_{N}, \theta_0, \ldots, \theta_{k-1})
\end{equation*}
where $m_{-1} = 1$.
The following statement follows from a standard calculation:
\begin{lem}
	The maps $\underline{d}$, $\underline{b}$ and $\underline{B}$ are
	well-defined, pairwise anticommuting differentials.
\end{lem}
Let $\mathrm{C}_{+}(\Omega)$ be the subspace of $\mathrm{C}(\Omega)$
containing all elements of even
degree and $\mathrm{C}_{-}(\Omega)$ be the subspace containing all elements
of odd degree.
It follows that we have a $\Z_2$-graded complex
\begin{center}
	\begin{tikzcd}
		\mathrm{C}_{+}(\Omega) \arrow[rr,
		"\underline{d}+\underline{b}-\underline{B}",
		bend left] &  & \mathrm{C}_{-}(\Omega) \arrow[ll,
		"\underline{d}+\underline{b}-\underline{B}", bend left]
	\end{tikzcd}.
\end{center}
\subsection{The bar construction}\label{Section: bar construction}
\subsubsection{Bar chains}
The \textit{bar construction for $\Omega$} is defined as the vector
space
\begin{equation}\label{Bar construction}
	\mathrm{B}(\Omega) \coloneqq \bigoplus_{N=0}^{\infty} \Omega[1]^{\otimes
		N}
\end{equation}
with grading given by
\begin{equation*}
	\mathrm{B}_n(\Omega) = \bigoplus_{N=0}^{\infty} \bigoplus_{l_1 + \ldots +
		l_N = n + N} \Omega^{l_1} \otimes \ldots \otimes \Omega^{l_N}
\end{equation*}
for $n \in \Z$.
We will usually denote elements of $\mathrm{B}(\Omega)$ by $(\theta_1,
\ldots, \theta_N)$ instead of $\theta_1 \otimes \ldots \otimes \theta_N$.
On $\mathrm{B}(\Omega)$ we have two differentials of degree $+1$ given by
\begin{align*}
	d(\theta_1, \ldots, \theta_N) & \coloneqq -\sum_{k=1}^N (-1)^{n_{k-1}}
	(\theta_1, \ldots, d \theta_k, \ldots, \theta_N) \\
	b'(\theta_1, \ldots, \theta_N) & \coloneqq - \sum_{k=1}^{N-1} (-1)^{n_{k}}
	(\theta_1, \ldots, \theta_k \theta_{k+1}, \ldots, \theta_N)
\end{align*}
where $n_k \coloneqq \lvert \theta_1 \rvert + \ldots + \lvert \theta_k \rvert
-k$.
A calculation similar to the one for the (reduced) cyclic complex shows that
$d$ and $b'$ are anti-commuting
differentials turning $\mathrm{B}(\Omega)$ into a $\Z_2$-graded complex with
differential $d + b'$ if we reduce the grading modulo two.
\subsubsection{Algebra valued bar cochains}
For a $\Z_2$-graded algebra $L$ the \textit{space of $L$-valued bar
cochains on $\Omega$} is defined as $\Hom(\mathrm{B}(\Omega), L)$, the space
of linear maps from $\mathrm{B}(\Omega)$ to $L$.
By the universal property of the tensor product, every $L$-valued bar cochain
$l$ is given by a sequence $(l^{(N)})_{N \in \N_0}$ of multilinear maps
\begin{equation*}
	l^{(N)} \colon \underbrace{\Omega[1] \times \ldots \times \Omega[1]}_N
	\to L
\end{equation*}
where for $N = 0$ we may identify $l^{(0)}$ with an element of $L$.
We turn the space of $L$-valued bar cochains into a $\Z_2$-graded algebra with
product
\begin{equation}\label{Product bar cochains}
	l_1 l_2 (\theta_1, \ldots, \theta_N) \coloneqq \sum_{k=0}^N (-1)^{\lvert
		l_2 \rvert n_k} l_1(\theta_1, \ldots, \theta_k) l_2 (\theta_{k+1},
	\ldots,
	\theta_N)
\end{equation}
and $\Z_2$-grading given by the usual grading on $\Hom(V, W)$ for
$\Z_2$-graded
vector spaces $V$ and $W$.
The \textit{codifferential $\delta$} on $\Hom(\mathrm{B}(\Omega), L)$ is
defined via
\begin{equation*}
	\delta \coloneqq - (d+b')^{\vee} \colon \Hom(\mathrm{B}(\Omega), L) \to
	\Hom(\mathrm{B}(\Omega), L).
\end{equation*}
A standard calculation shows the following
\begin{lem}
	$\delta$ fulfills the graded Leibniz rule, i. e.
	\begin{equation}\label{Z2 graded Leibniz delta}
		\delta(l_1 l_2) = (\delta l_1)l_2 + (-1)^{\lvert l_1 \rvert} l_1
		(\delta
		l_2).
	\end{equation}
	In particular, $\Hom(\mathrm{B}(\Omega), L)$ is a differential
	$\Z_2$-graded algebra.
\end{lem}
\subsection{Analytic cochains and entire chains}\label{Section: entire}
Similarly to the previous section, we define the \textit{space of cyclic
complex valued cochains} to be the space $\Hom(\mathrm{C}(\Omega), \C)$ of
linear maps from $\mathrm{C}(\Omega)$ to $\C$.
To each cyclic cochain corresponds a
sequence $(l^{(N)})_{N \in \N_0}$ of multilinear maps
\begin{equation*}
	l^{(N)} \colon \Omega \times \underbrace{\underline{\Omega}[1] \times
		\ldots \times \underline{\Omega}[1]}_N \to \C.
\end{equation*}
\begin{defn}[\protect{\cite[Definition 3.1]{GüneysuLudewig}}]\label{Analytic}
	We call a cochain $l \in \Hom(\mathrm{C}(\Omega), \C)$ \textit{analytic}
	if there
	exists a continuous seminorm $\nu$ on $\Omega$ and a constant $C > 0$ such
	that for all $N \in \N_0$ and $\theta_0, \ldots, \theta_N \in \Omega$ we
	have
	\begin{equation*}
		\lvert l^{(N)}(\theta_0, \ldots, \theta_N) \rvert \leq
		\frac{C}{\lfloor N/2 \rfloor!} \, \nu(\theta_0) \cdot \ldots \cdot
		\nu(\theta_N).
	\end{equation*}
	We denote the space of analytic cochains by $\mathrm{C}_{\alpha}(\Omega)$.
\end{defn}
Recall that for a continuous seminorm $\nu$ on $\Omega$ the
\textit{$N$-th induced projective tensor seminorm $\pi_{\nu, N}$} on $\Omega
\otimes \underline{\Omega}[1]^{\otimes N}$ is given by
\begin{equation*}
	\pi_{\nu, N}(c_N) = \inf\left\{\sum_{\gamma} \nu(\theta_0^{(\gamma)})
	\cdot
	\ldots \cdot \nu(\theta_N^{(\gamma)}) : c_N = \sum_{\gamma}
	\theta_0^{(\gamma)} \otimes \ldots \otimes \theta_N^{(\gamma)}\right\},
\end{equation*}
the infimum being taken over all realizations of $c_N \in \mathrm{C}(\Omega)$
as a finite sum as in the above display.
If $l$ is an analytic cochain, it follows immediately from the definition
above that all induced linear maps $l^{(N)}$ are continuous on the projective
tensor product $\Omega \otimes
\underline{\Omega}[1]^{\otimes N}$.
Moreover, since multiplication and the differential $d$ are assumed to be
continuous on $\Omega$, it follows that the codifferentials
$\underline{d}^\vee$,
$\underline{b}^\vee$ and $\underline{B}^\vee$ preserve the space
$\mathrm{C}_{\alpha}(\Omega)$.
\begin{defn}[\protect{\cite[Definition 3.2]{GüneysuLudewig}}]
	For a continuous seminorm $\nu$ on $\Omega$ the seminorm
	$\epsilon_\nu$ on $\mathrm{C}(\Omega)$ is given by
	\begin{equation}\label{Definition entire seminorm}
		\epsilon_{\nu}(c) \coloneqq \sum_{N = 0}^\infty \frac{\pi_{\nu,
				N}(c_N)}{\lfloor N/2 \rfloor!}, \qquad c = \sum_{N =
			0}^\infty c_N \in
		\mathrm{C}(\Omega), \quad \text{where $c_N \in \Omega \otimes
			\underline{\Omega}[1]^{\otimes N}$}.
	\end{equation}
	Note that this is well-defined since $c$ is a finite sum.
	Then, the \textit{space $\mathrm{C}^{\epsilon}(\Omega)$ of entire chains}
	is defined as the
	completion of $\mathrm{C}(\Omega)$ with respect to the seminorms
	$\epsilon_{\nu}$.
\end{defn}
If $l$ is an analytic cochain, it follows from the definitions that it is
continuous with respect to the seminorms $\epsilon_{\nu}$ and hence extends
to a continuous functional on $\mathrm{C}^\epsilon(\Omega)$.
Also, since $\underline{d}$, $\underline{b}$ and $\underline{B}$ are
continuous
on $\mathrm{C}(\Omega)$ with respect to the seminorms $\epsilon_\nu$, they
extend to differentials on $\mathrm{C}^\epsilon(\Omega)$ which are odd with
respect to the induced $\Z_2$-grading.
Hence, we have defined the \textit{entire (reduced) cyclic complex of
	$\Omega$}.
\subsection{Chen normalization}\label{Section: chen normalization}
Similarly to \cite{GetzlerJonesPetrack}, one often considers a quotient of the
reduced cyclic complex $\mathrm{C}(\Omega_{\mathbb{T}})$ by some subcomplex
$\mathrm{D}^\T(\Omega)$ since some of the chains which are of interest are
not closed in the larger complex.

To this end, let
\begin{equation*}
	S_i^{(f)} \colon \mathrm{C}(\Omega_{\mathbb{T}}) \to
	\mathrm{C}(\Omega_{\mathbb{T}})
\end{equation*}
be the continuous linear map given by
\begin{equation}\label{Si f}
	S_i^{(f)}(\theta_0, \dots, \theta_N) =
	\begin{cases}
		(-1)^{m_i}(\theta_0, \ldots, \theta_{i}, f, \theta_{i+1}, \ldots,
		\theta_N), & 0 \leq i \leq N - 1 \\
		(-1)^{m_N}(\theta_0, \ldots, \theta_N, f), & i = N \\
		0, & \text{otherwise},
	\end{cases}
\end{equation}
for all $f \in \Omega^0$ and $i\in \N_0$, and let
\begin{equation*}
	T_i^{(f)} \colon \mathrm{C}(\Omega_{\mathbb{T}}) \to
	\mathrm{C}(\Omega_{\mathbb{T}})
\end{equation*}
be the continuous linear map defined via
\begin{align*}
	T_i^{(f)}(\theta_0, \ldots, \theta_N) =
	\begin{cases}
		\begin{aligned}[b]
			& (\theta_0, \ldots, \theta_i, f \theta_{i+1}, \ldots, \theta_N)
			\\
			& \quad -(\theta_0, \ldots, \theta_i f, \theta_{i+1}, \ldots,
			\theta_N)
			\\
			& \quad -(\theta_0, \ldots, \theta_i, df, \theta_{i+1}, \ldots,
			\theta_N)
		\end{aligned}, & \quad 0 \leq i \leq N - 1 \\
		\begin{aligned}[b]
			& (f\theta_0, \ldots, \theta_N) \\
			& \quad -(\theta_0, \ldots, \theta_N f) \\
			& \quad -(\theta_0, \ldots, \theta_N, df)
		\end{aligned}, & \quad i = N
	\end{cases}
\end{align*}
for all $f \in \Omega^0$ and $i\in \N_0$.
Note that $\Omega^0 \subseteq \Omega_\T^0$ is a proper subset in general.
Moreover, define
\begin{align}
	&S \colon \mathrm{C}(\Omega_{\mathbb{T}}) \to
	\mathrm{C}(\Omega_{\mathbb{T}}), \quad (\theta_0, \dots, \theta_N) \mapsto
	\sum_{k=0}^N (\theta_0, \dots, \theta_{k}, \sigma, \theta_{k+1}, \dots,
	\theta_N),\label{Chen normalization S}\\
	&R \colon \mathrm{C}(\Omega_{\mathbb{T}}) \to
	\mathrm{C}(\Omega_{\mathbb{T}}), \quad (\theta_0, \dots, \theta_N) \mapsto
	(\sigma \theta_0, \dots, \theta_N).\label{Chen normalization R}
\end{align}
Finally, we let $\mathrm{D}^{\mathbb{T}}(\Omega) \subseteq
\mathrm{C}(\Omega_{\mathbb{T}})$ be the smallest subcomplex containing the
images of the operators
\begin{equation}\label{Chen operators}
	S + \mathbf{1}, \quad R, \quad S_{i}^{(f)}, \quad T_{i}^{(f)}
\end{equation}
for all $i \in \N_0$ and $f \in \Omega^0$.
\begin{defn}[\protect{\cite[Definition 3.4]{GüneysuLudewig}}]
	The quotient complex
	\begin{equation*}
		\mathrm{N}^{\T, \epsilon}(\Omega) \coloneqq
		\mathrm{C}^{\epsilon}(\Omega_\T) /
		\overline{\mathrm{D}^{\mathbb{T}}(\Omega)},
	\end{equation*}
	where the closure is taken inside $\mathrm{C}^\epsilon(\Omega_\T)$, is
	called the \textit{Chen normalized entire complex}.
	The space $\mathrm{N}_{\T, \alpha}(\Omega)$ of analytic cochains
	vanishing on $\mathrm{D}^{\mathbb{T}}(\Omega)$ will be referred to as the
	\textit{space of extended Chen normalized analytic cochains}.
	This induces $\Z_2$-graded complexes
	\begin{equation*}
		\begin{tikzcd}
			\mathrm{N}_{+}^{\T, \epsilon}(\Omega) \ar[rr,  bend left=20,
			"\underline{d}+\underline{b}-\underline{B}"] & &  \ar[ll,  bend
			left=20, "\underline{d}+\underline{b}-\underline{B}"]
			\mathrm{N}_{-}^{\T, \epsilon}(\Omega),
		\end{tikzcd}
		~~~~\text{and}~~~~
		\begin{tikzcd}
			\mathrm{N}_{\T, \alpha}^+(\Omega) \ar[rr,  bend left=20,
			"(\underline{d}+\underline{b}-\underline{B})^{\vee}"] & &  \ar[ll,
			bend
			left=20, "(\underline{d}+\underline{b}-\underline{B})^{\vee}"]
			\mathrm{N}_{\T, \alpha}^-(\Omega).
		\end{tikzcd}
	\end{equation*}
\end{defn}
\begin{rem}\label{Remark Chen generators}
	Let $V$ denote the linear subspace generated by the images of the operators
	in \cref{Chen operators}.
	Since the differential $\underline{d}_\T +
	\underline{b} - \underline{B}$ of $\mathrm{C}(\Omega_\T)$ (where
	$\underline{d}_\T$ is induced by $d_\T$) squares to zero,
	$\mathrm{D}^\T(\Omega) = V +
	(\underline{d}_\T + \underline{b} - \underline{B})V$.
	Hence, a
	coclosed cochain on $\mathrm{C}(\Omega_\T)$ belongs to
	$\mathrm{N}_{\T,\alpha}(\Omega)$ as soon as it vanishes on $V$.
\end{rem}
\section{\texorpdfstring{$\vartheta$}{ϑ}-summable
	\texorpdfstring{$\mathcal{C}_q$}{C(q)}-Fredholm modules}\label{Section:
	Cq FM}
For $q \in \N_0$ we denote by $\mathcal{C}_q$ the complex Clifford algebra
with generators $e_1, \ldots, e_q$ subject to the relations
\begin{equation*}
	e_i e_j + e_j e_i = -2 \delta_{ij}.
\end{equation*}
We turn $\mathcal{C}_q$ into a $*$-algebra by setting $e_i^* \coloneqq -e_i$
for
each $i$.
Moreover, $\mathcal{C}_q$ becomes a graded algebra by declaring the generators
$e_i$ to be odd.
Note that for $q = 0$ we have $\mathcal{C}_q = \C$.

Given a $\Z_2$-graded Hilbert space $\HH$ we denote by $\Lop(\HH)$ the
$\Z_2$-graded
$C^*$-algebra of bounded operators on $\HH$ where the involution is given by
the Hilbert space adjoint.
\begin{defn}
	A \textit{Hilbert module over $\mathcal{C}_q$} is a $\Z_2$-graded Hilbert
	space $\HH$
	together with a $\Z_2$-graded unital $*$-representation $\mathcal{C}_q \to
	\Lop(\HH)$.
\end{defn}
\begin{rem}
	\begin{enumerate}
		\item Note that for $q = 0$ this is nothing but a $\Z_2$-graded
		Hilbert space.
		\item If $\HH$ is a Hilbert module over $\mathcal{C}_q$, this
		canonically induces a Hilbert module structure on $\HH^m$ for $m \in
		\N$ by letting $\mathcal{C}_q$ act diagonally.
	\end{enumerate}
\end{rem}
\begin{nota}
	If $\HH$ is a Hilbert module over $\mathcal{C}_q$, we denote by
	$\Lop_{\mathcal{C}_q}(\HH)$ the subalgebra of $\Lop(\HH)$ consisting of
	operators on $\HH$ supercommuting with the $\mathcal{C}_q$-action.
	Furthermore, we define $\Lopclifqtrace(\HH)$ to be the subalgebra of
	trace-class operators supercommuting with the $\mathcal{C}_q$-action.
\end{nota}
\begin{defn}\label{Def FM}
	A \textit{$\vartheta$-summable $\mathcal{C}_q$-Fredholm module over a
	locally convex dg algebra $\Omega$} is a triple $\mathscr{M} = (\HH,
	\mathbf{c}, Q)$, where
	\begin{enumerate}[label=(\roman*)]
		\item $\HH$ is a Hilbert module over $\mathcal{C}_q$;
		\item $\mathbf{c} \colon \Omega \to \Lopclifq(\HH)$ is a
		continuous linear map which preserves the $\Z_2$-grading and
		satisfies $\mathbf{c}(\mathbf{1}) = \mathbf{1}$,
		\item $Q$ is an odd, self-adjoint (possibly) unbounded operator on
		$\HH$ supercommuting with $\mathcal{C}_q$,
	\end{enumerate}
	such that the relations
	\begin{equation} \label{Multiplicativity}
		[Q, \mathbf{c}(f)] = \mathbf{c}(df), \quad \text{and} \quad
		\mathbf{c}(f\theta) = \mathbf{c}(f)\mathbf{c}(\theta), \quad
		\mathbf{c}(\theta f) = \mathbf{c}(\theta)\mathbf{c}(f)
	\end{equation}
	hold for all $f \in \Omega^0$ and $\theta \in \Omega$.
	Furthermore, the above data has to fulfill the following analytic
	requirements:
	\begin{enumerate}
		\item[(A1)] For all $\theta\in \Omega$ the operators $C_\pm(\theta)
		\coloneqq \Delta^{\pm1/2}\mathbf{c}(\theta)\Delta^{\mp1/2}$ are
		densely defined and bounded, where $\Delta \coloneqq Q^2 +
		\mathbf{1}$; moreover, the
		map $\theta \mapsto C_\pm(\theta)$ is continuous from $\Omega$
		to $\Lopclifq(\HH)$;
		\item[(A2)] For each $T>0$ one has $e^{-T Q^2}\in
		\Lopclifqtrace(\HH)$.
	\end{enumerate}
	If $\mathscr{M}$ satisfies all the above assumptions except
	\eqref{Multiplicativity}, we call $\mathscr{M}$ a
	\textit{$\vartheta$-summable weak
		$\mathcal{C}_q$-Fredholm module}.
\end{defn}
\begin{eg}
	For $q = 0$ a $\vartheta$-summable (weak) $\mathcal{C}_0$-Fredholm module
	is nothing else but an even $\vartheta$-summable (weak)
	Fredholm module in the sense of \cite[Definition 2.1]{GüneysuLudewig}.
\end{eg}
\begin{defn}\label{Acyclic extension FM}
	Let $\mathscr{M} = (\HH, \mathbf{c}, Q)$ be a weak $\vartheta$-summable
	$\mathcal{C}_q$-Fredholm module over a locally convex dg algebra $\Omega$.
	We extend $\mathscr{M}$ to a weak $\vartheta$-summable
	$\mathcal{C}_q$-Fredholm module $\mathscr{M}_\T = (\HH, \mathbf{c}_\T,
	Q)$ over $\Omega_\T$, the acyclic extension of $\Omega$ (\cref{Acyclic
		extension}), via
	\begin{equation*}
		\mathbf{c}_\T(\theta' + \sigma \theta'') \coloneqq
		\mathbf{c}(\theta').
	\end{equation*}
	We call this the \textit{acyclic extension of $\mathscr{M}$.}
\end{defn}
\begin{defn}
	An \textit{isomorphism} between two $\vartheta$-summable weak
	$\mathcal{C}_q$-Fredholm modules $\mathscr{M}_1 = (\HH_1, \mathbf{c}_1,
	Q_1)$ and $\mathscr{M}_2 = (\HH_2, \mathbf{c}_2, Q_2)$ over a locally
	convex dg algebra $\Omega$ consists of a unitary
	isomorphism $U \colon \HH_1 \to \HH_2$ of $\Z_2$-graded Hilbert spaces
	intertwining all structure maps, i. e.
	\begin{equation*}
		U(q v) = q (Uv) \quad \forall q \in \mathcal{C}_q, v \in \HH_1,
		\qquad
		\text{and} \qquad U(\mathbf{c}_1(\theta) v) =
		\mathbf{c}_2(\theta)(Uv)
		\quad \forall \theta \in \Omega, v \in \HH_1,
	\end{equation*}
	and $U (Q_1 v) = Q_2 (Uv)$ for all $v \in \mathrm{dom}(Q_1)$ with the
	implicit assumption that $U$ maps $\mathrm{dom}(Q_1)$ onto
	$\mathrm{dom}(Q_2)$.
\end{defn}
Finally, as in \cite[Definition 6.1]{GüneysuLudewig}, we introduce the notion
of a homotopy of $\vartheta$-summable $\mathcal{C}_q$-Fredholm modules:
\begin{defn}\label{Definition homotopy}
	A \textit{homotopy of $\vartheta$-summable (weak) $\mathcal{C}_q$-Fredholm
	modules over $\Omega$} is a family $\mathscr{M}^s = (\HH,
	\mathbf{c}^s,
	Q_s)$, $s \in [0,1]$, of $\vartheta$-summable (weak)
	$\mathcal{C}_q$-Fredholm modules over $\Omega$ (with fixed
	$\mathcal{C}_q$ action) such that:
	\begin{enumerate}
		\item[(H1)] In (A1) from \cref{Def FM} one can choose all seminorms
		to be independent of $s$ and for all $T>0$, one has
		\begin{equation*}
			\sup_{s\in [0,1]}\Tr(e^{-TQ_s^2})<\infty.
		\end{equation*}
		\item[(H2)] The operators $Q_s$ are defined on the same domain for
		all $s \in [0,1]$ and for all $h \in \mathrm{dom}(Q_s)$,
		the assignment $s \mapsto Q_s h$ is continuously differentiable from
		$[0,1]$ to $\HH$.
		Thus, the time derivative $\dot{Q}_s$ is densely defined on $\HH$.
		The operators $\dot{Q}_s\Delta_s^{-1/2}$ and
		$\Delta_s^{-1/2}\dot{Q}_s$, where $\Delta_s
		= Q_s^2+\mathbf{1}$, are assumed to be bounded in a uniform way:
		\begin{equation*}
			\sup_{s\in [0,1]}\bigl\|\Delta_s^{-1/2}\dot{Q}_s\bigr\| +
			\sup_{s\in
				[0,1]}\bigl\|\dot{Q}_s\Delta_s^{-1/2}\bigr\|<\infty.
		\end{equation*}
		\item[(H3)] The map $s\mapsto
		\mathbf{c}^{s}(\theta)$ defines a continuously differentiable curve
		for all $\theta\in \Omega$ if we equip $\Lop(\HH)$ with the strong
		operator topology.
		Moreover, for each $\theta\in \Omega$ and $s \in [0,1]$, the
		operators $\dot{C}^s_\pm(\theta) \coloneqq
		\Delta_s^{\pm1/2}\dot{\mathbf{c}}^s(\theta)\Delta_s^{\mp1/2}$ are
		densely defined and bounded and the assignment $\theta \mapsto
		\dot{C}^s_\pm(\theta)$ is continuous from $\Omega$ to
		$\Lopclifq(\HH)$ in such a way that the seminorms can be chosen
		independently of $s$.
	\end{enumerate}
\end{defn}
\section{Odd \texorpdfstring{$\vartheta$}{ϑ}-summable Fredholm
	modules}\label{Section: odd FM}
In this section, we will introduce the notion of an \textit{odd}
$\vartheta$-summable Fredholm module over a locally convex dg algebra and see
that this induces a $\vartheta$-summable $\mathcal{C}_1$-Fredholm module in a
canonical way.
\begin{defn}\label{Definition odd FM}
	An \textit{odd $\vartheta$-summable Fredholm module} $\mathscr{M}$ over a
	locally convex dg algebra $\Omega$ is
	a triple $(\HH, \mathbf{c}, Q)$, where
	\begin{enumerate}
		\item $\HH$ is a Hilbert space;
		\item $\mathbf{c} \colon \Omega \to \Lop(\HH)$ is a continuous, linear
		operator such that $\mathbf{c}(\mathbf{1}) = \mathbf{1}$;
		\item $Q$ is an unbounded, self-adjoint operator on $\HH$,
	\end{enumerate}
	such that the following holds for all $f \in \Omega^0$ and $\theta \in
	\Omega$:
	\begin{equation}\label{Equations Fredholm}
		[Q, \mathbf{c}(f)] = \mathbf{c}(df), \quad \text{and} \quad
		\mathbf{c}(f \theta) =
		\mathbf{c}(f) \mathbf{c}(\theta), \quad \mathbf{c}(\theta f) =
		\mathbf{c}(\theta) \mathbf{c}(f).
	\end{equation}
	Furthermore, the above data are subject to analytic requirements:
	\begin{enumerate}[label=(A{{\arabic*}})]
		\item For all $\theta \in \Omega$, the operators $C_{\pm}(\theta)
		\coloneqq \Delta^{\pm 1/2} \mathbf{c}(\theta) \Delta^{\mp 1/2}$ are
		densely defined and bounded where $\Delta \coloneqq Q^2 + \mathbf{1}$
		and the map $\theta \mapsto C_{\pm}(\theta)$ is continuous from
		$\Omega$ to $\Lop(\HH)$;
		\item For all $T > 0$, one has $e^{-TQ^2} \in \Lop^1(\HH)$.
	\end{enumerate}
	If $\mathscr{M}$ satisfies all properties apart from \cref{Equations
	Fredholm}, it is called an \textit{odd $\vartheta$-summable weak Fredholm
	module over $\Omega$}.
\end{defn}
Suppose we are given an odd $\vartheta$-summable Fredholm module
$\mathscr{M}$.
Then, we may define a $\vartheta$-summable $\mathcal{C}_1$-Fredholm module
$\tilde{\mathscr{M}} \coloneqq (\tilde{\HH}, \tilde{\mathbf{c}},
\tilde{Q})$ as follows:
\begin{enumerate}
	\item Let $\tilde{\HH} \coloneqq \HH \oplus \HH$ with
	$\Z_2$-grading determined by $\tilde{\HH}^+ \coloneqq \HH \oplus
	\{0\}$,
	$\tilde{\HH}^- \coloneqq \{0\} \oplus \HH$;
	\item Let $\mathcal{C}_1$ act on $\tilde{\HH}$ by letting the
	generator
	$e_1$ act on $\tilde{\HH}$ via
	\begin{equation*}
		\gamma \coloneqq
		\begin{pmatrix}
			0 & -\mathbf{1} \\
			\mathbf{1} & 0
		\end{pmatrix};
	\end{equation*}
	\item
	\begin{equation*}
		\tilde{\mathbf{c}}(\theta) \coloneqq
		\begin{cases}
			\begin{pmatrix}
				\mathbf{c}(\theta) & 0 \\
				0 & \mathbf{c}(\theta)
			\end{pmatrix}, & \text{$\theta$ even} \\[15pt]
			\begin{pmatrix}
				0 & \mathbf{c}(\theta) \\
				\mathbf{c}(\theta) &0
			\end{pmatrix}, & \text{$\theta$ odd}
		\end{cases} \qquad \text{for $\theta \in \Omega$};
	\end{equation*}
	\item
	\begin{equation*}
		\tilde{Q} \coloneqq
		\begin{pmatrix}
			0 & Q \\
			Q & 0
		\end{pmatrix}.
	\end{equation*}
\end{enumerate}
A simple computation shows that $\gamma$ anticommutes with $\tilde{Q}$ and with
$\tilde{\mathbf{c}}(\theta)$ for odd $\theta$ and commutes with
$\tilde{\mathbf{c}}(\theta)$ for even $\theta$, and that the relations
\eqref{Equations Fredholm} and the analytic requirements for $\mathscr{M}$
imply \eqref{Multiplicativity} and (A1), (A2) for $\tilde{\mathscr{M}}$.
\begin{rem}
	The same construction applies verbatim if $\mathscr{M}$ is an odd
	$\vartheta$-summable \textit{weak} Fredholm module; then
	$\tilde{\mathscr{M}}$ is a $\vartheta$-summable \textit{weak}
	$\mathcal{C}_1$-Fredholm module.
\end{rem}
\begin{rem}\label{Doubling acyclic}
	Consider the acyclic extension $\Omega_{\mathbb{T}}$ of
	$\Omega$.
	As in \cref{Acyclic extension FM}, we may define the extended odd
	$\vartheta$-summable weak Fredholm module $\mathscr{M}_{\mathbb{T}} =
	(\HH,\mathbf{c}_{\mathbb{T}}, Q)$ over $\Omega_{\mathbb{T}}$ by setting
	\begin{equation*}
		\mathbf{c}_{\mathbb{T}}(\theta' + \sigma \theta'') \coloneqq
		\mathbf{c}(\theta').
	\end{equation*}
	Comparing the definitions, it is obvious that
	$\widetilde{\mathscr{M}_{\mathbb{T}}} =
	\tilde{\mathscr{M}}_{\mathbb{T}}$.
\end{rem}
\section{The Chern character}\label{Section: Chern character}
Similarly to \cite{GüneysuLudewig}, for a $\vartheta$-summable weak
$\mathcal{C}_q$-Fredholm module $\mathscr{M}$, we will construct an analytic
cochain $\Chern$ which we will call the \textit{Chern character of
	$\mathscr{M}$}.
\subsection{The quantization map}
Let $\mathscr{M} \coloneqq (\mathcal{H}, \mathbf{c}, Q)$ be a
$\vartheta$-summable weak $\mathcal{C}_q$-Fredholm module.
Similarly to \cite[Equation (4.1)]{GüneysuLudewig}, we define an
$\Lopclifq(\HH)$-valued bar cochain $\omega_{\mathscr{M}}$ via
\begin{equation*}
	\omega_{\mathscr{M}}^{(0)} \coloneqq Q, \quad
	\omega_{\mathscr{M}}^{(1)}(\theta) \coloneqq \mathbf{c}(\theta), \quad
	\omega_{\mathscr{M}}^{(N)}(\theta_1, \ldots, \theta_N) \coloneqq 0, N \geq
	2.
\end{equation*}
The cochain $\omega_{\mathscr{M}}$ is odd in the $\Z_2$-grading of
$\Hom(\mathrm{B}(\Omega), \Lop_{\mathcal{C}_q}(\HH))$.
Note that since $Q$ may be unbounded in general (hence not actually an
element of
$\Lop_{\mathcal{C}_q}(\HH)$), the above display is only formal.
We may interpret $\omega_{\mathscr{M}}$ as a (super-)connection form of the
(super-)connection $\nabla \coloneqq \delta \omega_{\mathscr{M}} +
\omega_{\mathscr{M}}$ and define its curvature via
\begin{equation*}
	F_{\mathscr{M}} \coloneqq \delta \omega_{\mathscr{M}} +
	\omega_{\mathscr{M}}^2.
\end{equation*}

Working out the components, $F_{\mathscr{M}}$ is explicitly given by
\begin{align}\label{Components curvature}
	F_{\mathscr{M}}^{(0)} & = Q^2, \notag\\
	F_{\mathscr{M}}^{(1)}(\theta_1) & = [Q, \mathbf{c}(\theta_1)] - \mathbf{c}(d
	\theta_1), \\
	F_{\mathscr{M}}^{(2)}(\theta_1, \theta_2) & = (-1)^{\lvert
		\theta_1 \rvert}(\mathbf{c}(\theta_1 \theta_2) - \mathbf{c}(\theta_1)
	\mathbf{c}(\theta_2)),\notag
\end{align}
with all higher arities being zero.
\begin{nota}
	Let $H$ be a non-negative operator and $A_1, \ldots, A_N$ suitable
	operators on a Hilbert space $\HH$.
	Define
	\begin{equation*}
		\{A_1, \ldots, A_N\}_{H} \coloneqq \int_{\Delta_N} e^{-\tau_1 H} A_1
		e^{-(\tau_2 - \tau_1)H} A_2 \cdot \ldots \cdot e^{-(\tau_{N} -
			\tau_{N-1})H}A_N e^{-(1 - \tau_N)H} \, \dint \tau
	\end{equation*}
	where $\Delta_N\coloneqq \{\tau \in \R^N \mid 0 \leq \tau_1 \leq \ldots
	\leq \tau_N \leq 1\}$ is the standard simplex.
\end{nota}
\cite[Lemma 4.2]{GüneysuLudewig} shows that the components of the curvature
$F_{\mathscr{M}}$ satisfy the necessary analytic assumptions such that the
following definition of the \textit{quantization map} makes sense:
\begin{defn}\label{Defn Quantization}
	For $T>0$ the even $\Lop_{\mathcal{C}_q}(\HH)$-valued bar cochain
	$\Phi^{\mathscr{M}}_T$ is defined as
	\begin{equation*}
		\Phi^{\mathscr{M}}_T \coloneqq \sum_{N=0}^\infty (-T)^N
		\bigl\{\underbrace{F^{\geq 1}_{\mathscr{M}}, \dots, F^{\geq
		1}_{\mathscr{M}}}_N\bigr\}_{TQ^2}.
	\end{equation*}
\end{defn}
We recall:
\begin{nota}
	Let $\mathcal{P}_{M, N}$ be the set of ordered partitions of length $M$
	of $\{1, \ldots, N\}$; that is, $M$-tuples of nonempty subsets
	$I = (I_1, \ldots, I_M)$ for which $I_1 \cup \ldots \cup I_M = \{1, \dots,
	N\}$ and for which every member of $I_a$ is smaller than any member of
	$I_b$ if $a < b$.
\end{nota}
For $N \geq 1$, we have the formula (see \cite[Eq.
(4.8)]{GüneysuLudewig})
\begin{equation*}
	\Phi^{\mathscr{M}}_T(\theta_1, \dots, \theta_N) = \sum_{M=1}^N (-T)^M
	\sum_{I \in
		\mathcal{P}_{M, N}} \bigl\{ F_{\mathscr{M}}(\theta_{I_1}), \dots,
	F_{\mathscr{M}}(\theta_{I_M})\bigr\}_{TQ^2},
\end{equation*}
where $\theta_{I_a}:= (\theta_{i+1} , \dots , \theta_{i+m})$
if $I_a = \{j \mid i < j \leq i+m\}$ for some $i, m$.
For $N = 0$ we have $(\Phi_T^{\mathscr{M}})^{(0)} = e^{-TQ^2}$.
\begin{rem}\label{Remark isomorphism quantm}
	\begin{enumerate}
		\item We have that $\Phi \in \Hom^+(\mathrm{B}(\Omega),
		\Lop_{\mathcal{C}_q}(\HH))$, i. e. it maps even elements of
		$\mathrm{B}(\Omega)$ to even operators on $\HH$ and odd elements to
		odd operators, since $F_{\mathscr{M}}$ is even (because $\omega$ is
		odd).
		\item\label{Item 3} If $U \colon \mathscr{M}_1 \to \mathscr{M}_2$ is
		an isomorphism
		of $\vartheta$-summable weak $\mathcal{C}_q$-Fredholm modules, it is
		easy to check that
		\begin{equation*}
			U \Phi_T^{\mathscr{M}_1}(\theta) = \Phi_T^{\mathscr{M}_2}(\theta)
			U
		\end{equation*}
		for all $\theta \in \mathrm{B}(\Omega)$.
	\end{enumerate}
\end{rem}
The following fundamental estimate is proved exactly as
\cite[Theorem 4.6]{GüneysuLudewig}: the proof only uses operator and trace
norm estimates for the brackets, to which the Clifford action is irrelevant.
\begin{thm}\label{Fundamental estimate}
	The operator $\Phi_{T}^{\mathscr{M}}(\theta_1, \ldots, \theta_N)$ is
	well-defined and trace-class for all $T > 0$ and $\theta_1, \ldots,
	\theta_N \in \Omega$.
	There exists a continuous seminorm $\nu$ on $\Omega$, which is independent
	of $T$, such that
	\begin{equation*}
		\lVert \Phi_{T}^{\mathscr{M}}(\theta_1, \ldots, \theta_N) \rVert_1
		\leq
		e^{T/2} \Tr(e^{-TQ^2/2}) \frac{T^{N/2}}{\lfloor N/2 \rfloor!}
		\nu(\theta_1) \cdot \ldots \cdot \nu(\theta_N).
	\end{equation*}
	Moreover, the same holds for $Q \Phi_{T}^{\mathscr{M}}$ and
	$\Phi_T^{\mathscr{M}} Q$ instead of $\Phi_T^{\mathscr{M}}$ if we replace
	$T^{N/2}$ with $T^{(N-1)/2}$ and $\lfloor N/2\rfloor!$ with $\lfloor
	(N-1)/2 \rfloor!$ on the right-hand side.
\end{thm}
\subsection{Definition of \texorpdfstring{$\Chern$}{ChernM}}
\begin{defn}\label{Clifford supertrace}
	Let $\HH$ be a Hilbert module over $\mathcal{C}_q$ and $A \in
	\Lopclifqtrace(\HH)$.
	The \textit{Clifford supertrace of $A$} is defined as
	\begin{equation*}
		\CStr(A) \coloneqq 2^{-q}\Str(\gamma A)
	\end{equation*}
	where $\gamma \coloneqq e_1 \cdot \ldots \cdot e_q$ is the Clifford volume
	element.
\end{defn}
The following lemma justifies the name \enquote{Clifford supertrace}:
\begin{lem}\label{Clifford trace is trace}
	Let $\HH$ be a Hilbert module over $\mathcal{C}_q$ and $A, B \in
	\Lopclifqtrace(\HH)$.
	Then one has
	\begin{equation*}
		\CStr([A, B]) = 0.
	\end{equation*}
\end{lem}
\begin{proof}
	The volume element $\gamma$ has parity $q$, and since $B$ supercommutes
	with each $e_i$, we have $B\gamma = (-1)^{q\lvert B\rvert}\gamma B$. Using
	the graded cyclicity of the supertrace,
	\begin{equation*}
		\Str(\gamma A B) = (-1)^{(q + \lvert A \rvert)\lvert B\rvert}
		\Str(B\gamma A) = (-1)^{\lvert A\rvert\lvert B\rvert}\Str(\gamma B
		A),
	\end{equation*}
	which is the claim.
\end{proof}
Note also that $\gamma$ is unitary, so that $\lvert \CStr(A)\rvert \leq
\lVert A\rVert_1$ for all $A \in \Lopclifqtrace(\HH)$.
With the above definitions in place, we may define the Chern character:
\begin{defn}
	Let $\mathscr{M}$ be a $\vartheta$-summable weak $\mathcal{C}_q$-Fredholm
	module.
	The \textit{Chern character of $\mathscr{M}$} is the cochain $\Chern \in
	\Hom(\mathrm{C}(\Omega), \C)$ defined via
	\begin{equation*}\label{Definition Chern character of FM}
		\Chern(\theta_0, \ldots, \theta_N) \coloneqq
		\CStr(\mathbf{c}(\theta_0)
		\Phi_1^{\mathscr{M}}(\theta_1, \ldots, \theta_N))
	\end{equation*}
	for $\theta_0, \ldots, \theta_N \in \Omega$.
\end{defn}
\begin{rem}
	For $q = 0$ this coincides with \cite[Definition 5.1]{GüneysuLudewig}.
\end{rem}
\begin{lem}\label{Chern even/odd}
	$\Chern \in \Hom^+(\mathrm{C}(\Omega), \C)$ for $q$ even and $\Chern \in
	\Hom^-(\mathrm{C}(\Omega), \C)$ for $q$ odd.
\end{lem}
\begin{proof}
	For $q$ even we have to prove that $\Chern$ is even, i. e. it vanishes on
	odd elements of $\mathrm{C}(\Omega)$.
	Let $(\theta_0, \ldots, \theta_N) \in \mathrm{C}(\Omega)$ be an odd
	element.
	There are two cases to consider:
	\begin{enumerate}
		\item $\theta_0$ is even in $\Omega$.
		By definition of the gradings on $\mathrm{C}(\Omega)$ and
		$\mathrm{B}(\Omega)$ this implies that $(\theta_1, \ldots, \theta_N)$
		is odd in $\mathrm{B}(\Omega)$.
		Since $\mathbf{c}$ and $\Phi_1^{\mathscr{M}}$ are parity preserving,
		it
		follows that
		\begin{equation*}
			\mathbf{c}(\theta_0) \Phi_1^{\mathscr{M}}(\theta_1, \ldots,
			\theta_N)
		\end{equation*}
		is an odd operator on $\HH$.
		Thus, as $\gamma$ is even (because $q$ is), $\gamma
		\mathbf{c}(\theta_0)\Phi_1^{\mathscr{M}}(\theta_1, \ldots, \theta_N)$
		is an odd operator on which the usual supertrace on $\HH$ vanishes.
		\item $\theta_0$ is odd in $\Omega$.
		In this case, the same argument shows that $(\theta_1, \ldots,
		\theta_N)$ must be even in $\mathrm{B}(\Omega)$ and the above proof
		goes through in a similar fashion.
	\end{enumerate}
	The case in which $q$ is odd may be proven analogously.
\end{proof}
\begin{lem}\label{Isomorphism Chern character}
	If $\mathscr{M}_1$ and $\mathscr{M}_2$ are isomorphic
	$\vartheta$-summable weak $\mathcal{C}_q$-Fredholm modules, then
	$\mathrm{Ch}_{\mathscr{M}_1} = \mathrm{Ch}_{\mathscr{M}_2}$.
\end{lem}
\begin{proof}
	This follows from \crefdefpart{Remark isomorphism quantm}{Item
		3} and the fact that the isomorphism intertwines the
	$\mathcal{C}_q$-actions.
\end{proof}
We collect some of the most important properties of the Chern character:
\begin{thm}\label{Thm ChM}
	Let $\mathscr{M}$ be a $\vartheta$-summable weak $\mathcal{C}_q$-Fredholm
	module.
	Then the following assertions hold:
	\begin{enumerate}
		\item The cochain $\Chern$ is analytic.
		\item The Chern character is coclosed, that is $(\underline{d} +
		\underline{b} - \underline{B})^\vee \Chern = 0$.
		\item If $\mathscr{M}$ satisfies \cref{Multiplicativity}, i. e. it is
		a $\vartheta$-summable $\mathcal{C}_q$-Fredholm module, and
		$\mathscr{M}_\T$ denotes its acyclic extension, then $\ChernT$ is Chen
		normalized, i. e. it vanishes
		on the subcomplex $\mathrm{D}^{\mathbb{T}}(\Omega)$ defined in
		\cref{Section: chen normalization}.
		\item For a homotopy $\mathscr{M}^s$, $s \in [0,1]$, of
		$\vartheta$-summable $\mathcal{C}_q$-Fredholm modules
		the Chern characters $\mathrm{Ch}_{\mathscr{M}^0_\T}$ and
		$\mathrm{Ch}_{\mathscr{M}^1_\T}$ are cohomologous in
		$\mathrm{N}_{\T, \alpha}(\Omega)$.
	\end{enumerate}
\end{thm}
\subsubsection{Proof of \texorpdfstring{\cref{Thm ChM}}{Theorem 5.10}, Part 1}
Analyticity of $\Chern$ follows from the fundamental estimate in
\cref{Fundamental estimate} together with the fact that
\begin{equation*}
	\lVert AB \rVert_1 \leq \lVert A \rVert \cdot \lVert B \rVert_1
\end{equation*}
for bounded operators $A$ and $B$ with $B$ trace-class, and
$\lvert\CStr(A)\rvert \leq \lVert A\rVert_1$.\qed
\subsubsection{Proof of \texorpdfstring{\cref{Thm ChM}}{Theorem 5.10}, Part 2}
We will adapt the proof of \cite[Theorem 5.3]{GüneysuLudewig} to
our more general setting and highlight only the main differences.

The inclusion $\Omega \subseteq \Omega_{\mathbb{T}}$ gives rise to a linear
map
\begin{equation*}
	j \colon \mathrm{C}(\Omega) \to \mathrm{C}(\Omega_{\mathbb{T}})
\end{equation*}
which commutes with all differentials since $\Omega \subseteq
\Omega_{\mathbb{T}}$ is a subalgebra and $d_{\mathbb{T}}$ restricts to $d$ on
$\Omega$.
Observe that by definition of the acyclic extension of $\mathscr{M}$, we have
$\Chern = j^*\ChernT$ so that it suffices to prove coclosedness of $\ChernT$.

The maps $\Phi_1^{\mathscr{M}_{\mathbb{T}}}$ and
$F_{\mathscr{M}_{\mathbb{T}}}$ descend to the quotient
$\mathrm{B}(\underline{\Omega_{\mathbb{T}}})$ which is defined as in \cref{Bar
construction} with the graded vector space $\underline{\Omega_{\mathbb{T}}}$
instead of the dg algebra $\Omega_{\mathbb{T}}$; the quotient map of
$\Phi_1^{\mathscr{M}_{\mathbb{T}}}$ will be denoted by
$\underline{\Phi}_1^{\mathscr{M}_{\mathbb{T}}}$.
Using the explicit formulas from \cref{Components curvature} we
have
\begin{align*}
	F_{\mathscr{M}_{\mathbb{T}}}^{(0)} & = Q^2, \\
	F_{\mathscr{M}_{\mathbb{T}}}^{(1)}(\theta_1) & =
	[Q, \mathbf{c}(\theta_1')] - \mathbf{c}(d
	\theta_1') + \mathbf{c}(\theta_1''), \\
	F_{\mathscr{M}_{\mathbb{T}}}^{(2)}(\theta_1, \theta_2) & =
	(-1)^{\lvert
		\theta_1 \rvert}(\mathbf{c}(\theta_1' \theta_2') -
	\mathbf{c}(\theta_1')\mathbf{c}(\theta_2'))
\end{align*}
for $\theta_1, \theta_2 \in \Omega_{\mathbb{T}}$.
Let us introduce some notation:
define
\begin{equation*}
	\alpha \colon \mathrm{C}(\Omega_{\mathbb{T}}) \to
	\mathrm{B}(\underline{\Omega_{\mathbb{T}}}), \quad (\theta_0, \ldots,
	\theta_N) \mapsto \underline{\mathbf{N}}(\sigma \theta_0, \theta_1,
	\ldots,
	\theta_N)
\end{equation*}
to be the parity preserving map where
\begin{equation}\label{Definition alpha}
	\underline{\mathbf{N}} \colon \mathrm{B}(\underline{\Omega_{\mathbb{T}}})
	\to \mathrm{B}(\underline{\Omega_{\mathbb{T}}})
\end{equation}
is the quotient map of
\begin{equation*}
	\mathbf{N} \colon \mathrm{B}(\Omega_{\mathbb{T}}) \to
	\mathrm{B}(\Omega_{\mathbb{T}}), \quad (\theta_1, \ldots, \theta_N)
	\mapsto
	\sum_{k = 1}^N (-1)^{n_k(n_N - n_k)} (\theta_{k+1}, \ldots, \theta_N,
	\theta_1, \ldots, \theta_k)
\end{equation*}
and $n_k = \lvert \theta_1 \rvert + \ldots + \lvert \theta_k \rvert - k$.
\begin{lem}\label{Ch pullback}
	The Chern character of the acyclic extension satisfies
	\begin{equation*}
		\mathrm{Ch}_{\mathscr{M}_{\mathbb{T}}} = - \alpha^*
		\CStr(\underline{\Phi}_1^{\mathscr{M}_{\mathbb{T}}}).
	\end{equation*}
\end{lem}
We first record a cyclicity property of the Clifford supertrace on the
brackets $\{\ldots\}_H$ which will be used in the proof of \cref{Ch pullback}
below and again in \cref{Section: spectral flow pairing 2}.
\begin{lem}\label{Clifford trace cyclic sum property}
	Let $\HH$ be a Hilbert module over $\mathcal{C}_q$, $H$ a
	non-negative self-adjoint even
	operator on $\HH$ and $A_0, \ldots, A_N$ as in \cite[Lemma
	4.2]{GüneysuLudewig} with
	the extra condition that $H$ and the $A_j$ supercommute with the
	Clifford action.
	Then, the Clifford supertrace has the following cyclic property:
	\begin{equation*}
		\sum_{j = 0}^N (-1)^{k_j}\CStr(\{A_{j+1}, \ldots, A_N, A_0,
		\ldots,
		A_j\}_H) =
		\CStr(A_0 \{A_1, \ldots, A_N\}_H)
	\end{equation*}
	where $k_j \coloneqq (\lvert A_0 \rvert + \ldots + \lvert A_j
	\rvert)(\lvert A_{j+1} \rvert + \ldots + \lvert A_N \rvert)$.
\end{lem}
\begin{proof}
	As a consequence of \cref{Clifford trace is trace} we have
	\begin{align*}
		& \CStr\big([A_0 e^{-\tau_1 H}A_1 e^{-(\tau_2 - \tau_1)H} \cdot
		\ldots
		\cdot
		e^{-(\tau_j
			- \tau_{j-1})H} A_j e^{-(\tau_{j+1} - \tau_j)H}, \\
		& \qquad A_{j+1} e^{-(\tau_{j+2} - \tau_{j+1}) H}A_{j+2}
		e^{-(\tau_{j+3} - \tau_{j+2})H} \cdot \ldots \cdot e^{-(\tau_N
			- \tau_{N-1})H} A_N e^{-(1-\tau_N)H}]\big) = 0.
	\end{align*}
	Thus, after a substitution in $\Delta_N$, we have
	\begin{equation*}
		\CStr(A_0 \{A_1, \ldots, A_N\}_H) = (-1)^{k_j}\CStr(A_{j+1}
		\{A_{j+2},
		\ldots,
		A_N, A_0, \ldots, A_j\}_H)
	\end{equation*}
	for any $0 \leq j \leq N-1$ and in particular,
	\begin{equation}\label{Trace cyclic with 1}
		\CStr(A_0 \{A_1, \ldots, A_j, \mathbf{1}, A_{j+1}, \ldots,
		A_N\}_H) =
		(-1)^{k_j}\CStr(\{A_{j+1}, \ldots, A_N,
		A_0, \ldots, A_j\}_H).
	\end{equation}
	\begin{claim}
		For $0 \leq j \leq N$ we have
		\begin{equation}\label{Pulling out 1}
			\{A_1, \ldots, A_j, \mathbf{1}, A_{j+1}, \ldots, A_N\}_H =
			\int_{\Delta_N} (\tau_{j+1} - \tau_j) e^{-\tau_1 H}
			\prod_{i=1}^{N}
			A_i e^{-(\tau_{i+1} - \tau_i)H} \, \dint \tau
		\end{equation}
		where $\tau_{N+1} = 1$ and $\tau_0 = 0$.
	\end{claim}
	\begin{proof}[Proof of Claim]
		Let $t \in
		\Delta_{N+1}$ be the integration variables on the left-hand side, so
		that $\mathbf{1}$ sits between
		$e^{-(t_{j+1}-t_j)H}$ and $e^{-(t_{j+2}-t_{j+1})H}$.
		Using the semigroup property this simplifies to $e^{-(t_{j+2}-t_j)H}$,
		so the integrand does not depend on $t_{j+1}$.
		For $\tau \coloneqq
		(t_1, \ldots, t_j, t_{j+2}, \ldots, t_{N+1}) \in \Delta_N$, the
		variable $t_{j+1}$ ranges over $[\tau_j, \tau_{j+1}]$, and Fubini's
		theorem yields \cref{Pulling out 1}.
	\end{proof}
	Since $\sum_{j=0}^{N} (\tau_{j+1} - \tau_j) = 1$, this yields
	\begin{equation*}
		\sum_{j = 0}^N \{A_1, \ldots, A_j, \mathbf{1}, A_{j+1}, \ldots,
		A_N\}_H
		= \{A_1, \ldots, A_N\}_H
	\end{equation*}
	which in combination with \cref{Trace cyclic with 1} proves the
	statement.
\end{proof}
\begin{proof}[Proof of \cref{Ch pullback}]
	If $N = 0$, the statement is obvious.
	For $N \geq 1$, we calculate
	\begin{align*}
		& \alpha^*
		\CStr(\underline{\Phi}_1^{\mathscr{M}_{\mathbb{T}}})
		(\theta_0, \ldots, \theta_N) \\
		& \qquad = \sum_{k = 1}^{N+1} (-1)^{m_{k-1}(m_N - m_{k-1})}\CStr(
		\underline{\Phi}_1^{\mathscr{M}_{\mathbb{T}}}(\theta_k, \ldots,
		\theta_N, \sigma \theta_0, \theta_1, \ldots, \theta_{k-1}))
	\end{align*}
	where $m_k = \lvert \theta_0 \rvert + \ldots +
	\lvert \theta_k \rvert - k$.
	Indeed, the signs in the above equation follow from
	\begin{equation*}
		n_k(\sigma \theta_0, \ldots, \theta_N) = \lvert \sigma \theta_0 \rvert
		+ \lvert \theta_1 \rvert + \ldots + \lvert \theta_{k-1} \rvert - k =
		m_{k-1} \mod 2.
	\end{equation*}
	Using the explicit formulas for the curvature, one notices that
	\begin{equation*}
		F_{\mathscr{M}_{\mathbb{T}}}(\theta_N, \sigma\theta_0) =
		F_{\mathscr{M}_{\mathbb{T}}}(\sigma \theta_0, \theta_1) = 0
		\qquad \text{and} \qquad
		F_{\mathscr{M}_{\mathbb{T}}}(\sigma\theta_0) = \mathbf{c}(\theta_0').
	\end{equation*}
	Hence, expanding $\underline{\Phi}_1^{\mathscr{M}_{\mathbb{T}}}$ in
	terms of the brackets yields
	\begin{equation*}
		\sum_{k=1}^{N+1} (-1)^{m_{k-1}(m_N-m_{k-1})} \sum_{M=1}^N (-1)^{M+1}
		\sum_{l=1}^{M+1} \sum_{I \in \mathcal{P}_{M, N}^{k, l}}
		\begin{aligned}
			\CStr &\bigl\{F_{\mathscr{M}_{\mathbb{T}}}
			(\theta_{I_{l}}), \dots \\
			& \dots, F_{\mathscr{M}_{\mathbb{T}}}(\theta_{I_M}),
			\mathbf{c}(\theta_0^\prime),
			F_{\mathscr{M}_{\mathbb{T}}}(\theta_{I_1}), \dots \\
			&~~~~~~~~~~~~~~~~~~~~~\dots,
			F_{\mathscr{M}_{\mathbb{T}}}(\theta_{I_{l-1}})
			\bigr\}_{Q^2},
		\end{aligned}
	\end{equation*}
	where $\mathcal{P}^{k, l}_{M, N} \subseteq \mathcal{P}_{M, N}$ are those
	partitions for which $I_1, \dots, I_{l-1} \subseteq \{1, \dots,
	k-1\}$ and $I_{l}, \dots, I_M \subseteq \{k, \dots, N\}$.
	Let $I \in \mathcal{P}^{k, l}_{M, N}$ such that the bracket in the above
	display is non-zero and let $1 \leq j \leq l-1$.
	On the one hand, if $\theta_{I_j} = (\theta_{i_j}, \theta_{i_{j}+1})$, we
	have
	\begin{align*}
		\lvert F_{\mathscr{M}_{\mathbb{T}}}(\theta_{I_j}) \rvert & =
		\lvert \mathbf{c}(\theta_{i_j}' \theta_{i_j + 1}') -
		\mathbf{c}(\theta_{i_j}')\mathbf{c}(\theta_{i_j + 1}') \rvert =
		\lvert \theta_{i_j}' \rvert + \lvert \theta_{i_j + 1}' \rvert \\
		& = \lvert \theta_{i_j} \rvert + \lvert \theta_{i_j + 1} \rvert - 2
		\quad \mod 2,
	\end{align*}
	while on the other hand, if $\theta_{I_j} = \theta_{i_j}$, we have
	\begin{equation*}
		\lvert F_{\mathscr{M}_{\mathbb{T}}}(\theta_{I_j}) \rvert =
		\lvert [Q, \mathbf{c}(\theta_{i_j}')] -
		\mathbf{c}(d
		\theta_{i_j}') + \mathbf{c}(\theta_{i_j}'') \rvert = \lvert
		\theta_{i_j}' \rvert + 1 = \lvert \theta_{i_j} \rvert - 1 \mod 2.
	\end{equation*}
	Thus, we have
	\begin{equation*}
		m_{k-1} = \lvert \theta_0' \rvert + \lvert
		F_{\mathscr{M}_{\mathbb{T}}}(\theta_{I_1}) \rvert + \ldots +
		\lvert F_{\mathscr{M}_{\mathbb{T}}}(\theta_{I_{l-1}}) \rvert
		\eqqcolon \tilde{m}_{l-1}
	\end{equation*}
	modulo two for $1 \leq k \leq N + 1$.
	Note that for all $l$ and $M \leq N$, the sets $\mathcal{P}_{M,
	N}^{k, l}$ and $\mathcal{P}_{M,
	N}^{k', l}$ are disjoint for $k \neq k'$ and their union over $k$ is
	$\mathcal{P}_{M, N}$.
	Thus, the above sums over $k$ and $\mathcal{P}_{M, N}^{k, l}$ may be
	written as one sum over $\mathcal{P}_{M, N}$, yielding
	\begin{equation}\label{Chern pullback long sum}
		- \sum_{M=1}^N (-1)^M \sum_{I \in \mathcal{P}_{M, N}} \sum_{l=1}^{M+1}
		(-1)^{\tilde{m}_{l-1}(\tilde{m}_M - \tilde{m}_{l-1})}
		\begin{aligned}
			\CStr
			&\bigl\{F_{\mathscr{M}_{\mathbb{T}}}(\theta_{I_{l}}), \dots
			\\
			& \dots, F_{\mathscr{M}_{\mathbb{T}}}(\theta_{I_M}),
			\mathbf{c}(\theta_0^\prime),
			F_{\mathscr{M}_{\mathbb{T}}}(\theta_{I_1}), \dots \\
			&~~~~~~~~~~~~~~~~~~~~~\dots,
			F_{\mathscr{M}_{\mathbb{T}}}(\theta_{I_{l-1}})
			\bigr\}_{Q^2}.
		\end{aligned}
	\end{equation}
	Since all operators in question supercommute with the Clifford action,
	we may apply \cref{Clifford trace cyclic sum property} to \cref{Chern
	pullback long sum} which yields
	\begin{equation*}
		\ldots = - \CStr(\mathbf{c}(\theta_0')
		\underline{\Phi}_1^{\mathscr{M}_{\mathbb{T}}}(\theta_1, \ldots,
		\theta_N)) = -
		\mathrm{Ch}_{\mathscr{M}_{\mathbb{T}}}(\theta_0,
		\ldots, \theta_N)
	\end{equation*}
	which proves the lemma.
\end{proof}
By definition of the graded dual of a map and due to \cref{Chern even/odd},
we have
\begin{equation*}
	(\underline{d}_{\mathbb{T}} + \underline{b} - \underline{B})^{\vee}
	\mathrm{Ch}_{\mathscr{M}_{\mathbb{T}}} = (-1)^q
	\mathrm{Ch}_{\mathscr{M}_{\mathbb{T}}} \circ
	(\underline{d}_{\mathbb{T}} + \underline{b} - \underline{B})
\end{equation*}
where $\underline{d}_{\mathbb{T}} \coloneqq \underline{d_{\T}}$ is the
first differential on $\mathrm{C}(\Omega_{\mathbb{T}})$ induced by $d_\T$.
In order to employ \cref{Ch pullback}, we have to calculate the composition
of $\alpha$ with the various differentials.
A standard calculation shows that\footnote{This is precisely \cite[Equation
(5.7)]{GüneysuLudewig} modulo the fact that there is a small
misprint in the equation concerning $\underline{d}_{\mathbb{T}}\alpha$ as
well as an ambiguity in the definition of $\underline{S}$.}
\begin{equation*}
	\underline{b}' \alpha = \alpha \underline{b}, \quad \alpha
	\underline{d}_{\mathbb{T}} = \underline{d}_{\mathbb{T}} \alpha - h,
	\quad \alpha \underline{B} = \underline{S}' h,
\end{equation*}
where in the first equation $\underline{b}'$ denotes the second differential
on the quotient
$\mathrm{B}(\underline{\Omega_{\mathbb{T}}})$, in the second equation on the
right hand side $\underline{d}_{\mathbb{T}}$ denotes the first
differential on $\mathrm{B}(\underline{\Omega_{\mathbb{T}}})$ and $h$ is
given by
\begin{equation*}
	h \colon \mathrm{C}(\Omega_{\mathbb{T}}) \to
	\mathrm{B}(\underline{\Omega_{\mathbb{T}}}), \quad (\theta_0, \ldots,
	\theta_N) \mapsto \underline{\mathbf{N}}(\theta_0, \ldots, \theta_N)
\end{equation*}
and $\underline{S}'$ is given by
\begin{equation*}
	\underline{S}'(\theta_1, \ldots, \theta_N) \coloneqq (\sigma,
	\theta_1,
	\ldots, \theta_N) + \underline{S}(\theta_1, \ldots, \theta_N)
\end{equation*}
where $\underline{S}$ is defined on
$\mathrm{B}(\underline{\Omega_{\mathbb{T}}})$ analogously to $S$ introduced
in \cref{Chen normalization S}.
This yields
\begin{equation}\label{Interaction alpha}
	\alpha(\underline{d}_{\mathbb{T}} + \underline{b} - \underline{B}) =
	(\underline{d}_{\mathbb{T}} + \underline{b}')\alpha - (\underline{S}' +
	\mathbf{1})h.
\end{equation}
and hence, due to \cref{Ch pullback}, we have
\begin{align*}
	(\underline{d}_{\mathbb{T}} + \underline{b} - \underline{B})^{\vee}
	\mathrm{Ch}_{\mathscr{M}_{\mathbb{T}}} &= (-1)^{q+1} \CStr\big(
	\underline{\Phi}_1^{\mathscr{M}_{\mathbb{T}}} \alpha
	(\underline{d}_{\mathbb{T}} + \underline{b} - \underline{B})\big) \\
	\notag
	& = (-1)^{q+1} \CStr\big(
	\underline{\Phi}_1^{\mathscr{M}_{\mathbb{T}}}
	(\underline{d}_{\mathbb{T}} + \underline{b}')\alpha\big) - (-1)^{q+1}
	\CStr\big(
	\underline{\Phi}_1^{\mathscr{M}_{\mathbb{T}}}(\underline{S}' +
	\mathbf{1})h\big).
\end{align*}
If $c \in \mathrm{B}(\underline{\Omega_{\mathbb{T}}})$, it follows from
\cite[Theorem 4.7]{GüneysuLudewig}
that $\underline{\Phi}_1^{\mathscr{M}_{\mathbb{T}}}
\left((\underline{d}_{\mathbb{T}} + \underline{b}')\alpha(c)\right)$ is a
sum of graded commutators of operators on $\HH$ on which, by virtue of
\cref{Clifford trace is trace}, the Clifford supertrace $\CStr$ vanishes.
Moreover, using that $F(\sigma) = \mathbf{1}$ and $F(\sigma, \theta) = 0
= F(\theta,
\sigma)$ for all $\theta \in \Omega_{\mathbb{T}}$, the same arguments
as used in \cref{Ch pullback} together with \cref{Pulling out 1} show that
\begin{equation}\label{Clifford supertrace and S + id}
	\CStr(
	\underline{\Phi}_1^{\mathscr{M}_{\mathbb{T}}}\underline{S}') =
	-
	\CStr(\underline{\Phi}_1^{\mathscr{M}_{\mathbb{T}}})
\end{equation}
when restricted to the subspace
$\mathrm{B}^{\sharp}(\underline{\Omega_{\mathbb{T}}})$ which is the image of
$\underline{\mathbf{N}}$.
Therefore,
\begin{equation*}
	\CStr\big(
	\underline{\Phi}_1^{\mathscr{M}_{\mathbb{T}}} (\underline{S}' +
	\mathbf{1})h\big) = 0
\end{equation*}
which proves coclosedness.\qed
\subsubsection{Proof of \texorpdfstring{\cref{Thm ChM}}{Theorem 5.10}, Part 3}
We need to prove that $\mathrm{Ch}_{\mathscr{M}_{\mathbb{T}}}$ vanishes on
the subcomplex $\mathrm{D}^{\mathbb{T}}(\Omega)$.
Since $\mathrm{Ch}_{\mathscr{M}_{\mathbb{T}}}$ is coclosed by Part 2, it
suffices by \cref{Remark Chen generators} to show that it vanishes on the
images of the operators in \cref{Chen operators}.
First of all, we notice that $\mathbf{c}(\sigma \theta_0) = 0$ so
that $\mathrm{Ch}_{\mathscr{M}_{\mathbb{T}}}$ vanishes on the image of $R$.
Moreover, since $\alpha(S + \mathbf{1}) = (S' + \mathbf{1})
\alpha$, it follows from \cref{Ch pullback} and \cref{Clifford supertrace
	and S + id} that $\mathrm{Ch}_{\mathscr{M}_{\mathbb{T}}}$ vanishes on the
image of $S + \mathbf{1}$.
Note that due to \cref{Multiplicativity}, for $f \in \Omega^0$ we have
\begin{equation*}
	F_{\mathscr{M}_\T}(f) = 0 \quad \text{and} \quad F_{\mathscr{M}_\T}(f,
	\theta) = F_{\mathscr{M}_\T}(\theta, f) = 0
\end{equation*}
for all $\theta \in \Omega_{\mathbb{T}}$, hence $\ChernT(S_i^{(f)}) = 0$.
Finally, we have to show that $\ChernT$ vanishes on the image of
$T_i^{(f)}$ for all $f \in \Omega^0$ and $i \in \N_0$.
However, this can be proven analogously to \cite[Theorem
5.6]{GüneysuLudewig}.\qed
\subsubsection{Proof of \texorpdfstring{\cref{Thm ChM}}{Theorem 5.10}, Part
4}
We will adapt the proof of \cite[Theorem 6.2]{GüneysuLudewig} to
our more general setting and highlight only the main differences.
To this end, let $\mathscr{M}^s = (\HH, \mathbf{c}^s, Q_s)$, $s\in [0,1]$, be
a homotopy of $\vartheta$-summable $\mathcal{C}_q$-Fredholm modules.
We define a family of
$\Lopclifqtrace(\HH)$-valued bar
cochains:
\begin{equation*}
	\Psi_T^{\mathscr{M}^s} \coloneqq -T\int_0^1
	{\Phi}_{uT}^{\mathscr{M}^s}\, \dot{\omega}_{\mathscr{M}^s}
	{\Phi}_{(1-u)T}^{\mathscr{M}^s} \,\dint u.
\end{equation*}
For improved readability, we will write $\Psi^s_T \coloneqq
\Psi_T^{\mathscr{M}^s}$, $\Phi_T^s \coloneqq
\Phi_T^{\mathscr{M}^s}$ etc.
\begin{lem}
	There exists a continuous seminorm $\nu$ on $\Omega$ such that for each
	$s \in [0,1]$ and $T > 0$ and all $\theta_1, \ldots, \theta_N \in \Omega$
	the operator $\Psi_T^{\mathscr{M}^s}(\theta_1, \ldots, \theta_N)$ is
	well-defined and trace-class with
	\begin{equation}\label{Estimate Psi}
		\lVert \Psi_T^s(\theta_1, \ldots, \theta_N) \rVert_1 \leq e^{T/2}
		\Tr(e^{-TQ_s^2/2}) \dfrac{T^{(N+1)/2}}{\lfloor N/2 \rfloor !}
		\nu(\theta_1) \cdot
		\ldots \cdot \nu(\theta_N).
	\end{equation}
	The same holds true for $\Psi_T^sQ_s$ and $Q_s \Psi_T^s$ instead
	of $\Psi_T^s$ if we replace
	$T^{(N+1)/2}$ with $T^{N/2}$ and $\lfloor N/2\rfloor!$ with $\lfloor
	(N-1)/2 \rfloor!$ on the right-hand side.
\end{lem}
\begin{proof}
	Similarly to \cref{Fundamental estimate}, it suffices to prove this for
	$\Psi_T^s$.
	We need the following
	\begin{claim}
		Let $H \geq 0$ be self-adjoint and let $A_1, \ldots, A_N, B$ be
		operators on $\HH$ for which the integrands of the brackets below are
		integrable in trace norm, as is the case under the assumptions of
		\cite[Lemma 4.2]{GüneysuLudewig}.
		Then, for every $1 \leq k \leq N+1$,
		\begin{align}\label{Bracket B}
			\begin{split}
				&\{A_1, \ldots, A_{k-1}, B, A_k, \ldots, A_N\}_{H} \\
				& ~~~~~~~~~~= \int_0^1 u^{k-1}(1-u)^{N-(k-1)}\{A_1, \ldots,
				A_{k-1}\}_{uH} B \{A_k, \ldots, A_N\}_{(1-u)H} \,
				\dint u
			\end{split}
		\end{align}
		where the empty bracket is $\{\,\}_{H} \coloneqq e^{-H}$.
	\end{claim}
	\begin{proof}[Proof of Claim]
		Let $\tau \in \Delta_{N+1}$ be integration variables of the bracket on
		the left-hand side and write its integrand as
		\begin{equation*}
			L(\tau)\, B\, R(\tau), \qquad \text{where} \qquad
			\begin{aligned}
				L(\tau) &\coloneqq e^{-\tau_1 H} A_1 e^{-(\tau_2 - \tau_1)H}
				\cdots A_{k-1} e^{-(\tau_k - \tau_{k-1})H}, \\
				R(\tau) &\coloneqq e^{-(\tau_{k+1} - \tau_k)H} A_k \cdots
				A_N e^{-(1 - \tau_{N+1})H}.
			\end{aligned}
		\end{equation*}
		We parametrize $\Delta_{N+1}$ by the map
		\begin{equation*}
			\Delta_{k-1} \times [0,1] \times \Delta_{N+1-k} \to
			\Delta_{N+1}, \quad (\sigma, u, \rho) \mapsto \tau, \quad
			\tau_j = \begin{cases}
				u \sigma_j, & j < k, \\
				u, & j = k, \\
				u + (1-u) \rho_{j-k}, & j > k.
			\end{cases}
		\end{equation*}
		A short calculation shows that $\dint\tau = u^{k-1}(1-u)^{N+1-k}\,
		\dint\sigma\, \dint\rho\, \dint u$ and
		\begin{alignat*}{2}
			\tau_j - \tau_{j-1} &= u\,(\sigma_j - \sigma_{j-1}), &\qquad
			& 1 \leq j \leq k, \\
			\tau_{k+j} - \tau_{k+j-1} &= (1-u)(\rho_j - \rho_{j-1}), &\qquad
			& 1 \leq j \leq N+2-k
		\end{alignat*}
		where we set $\tau_0 \coloneqq 0$, $\tau_{N+2} \coloneqq 1$, $\sigma_0 \coloneqq 0$,
		$\sigma_k \coloneqq 1$, $\rho_0 \coloneqq 0$ and $\rho_{N+2-k}\coloneqq 1$.
		Hence, as a function of $\sigma$, $L(\tau)$ is precisely the integrand
		defining $\{A_1, \ldots, A_{k-1}\}_{uH}$ and, as a function of $\rho$,
		$R(\tau)$ is precisely the integrand defining $\{A_k, \ldots,
		A_N\}_{(1-u)H}$.
		By change of variables and Fubini's theorem, it follows that
		\begin{multline*}
			\{A_1, \ldots, A_{k-1}, B, A_k, \ldots, A_N\}_{H} \\ = \int_0^1
			u^{k-1}(1-u)^{N+1-k} \Big(\int_{\Delta_{k-1}} L \,
			\dint\sigma\Big)\, B \,\Big(\int_{\Delta_{N+1-k}} R \,
			\dint\rho\Big)\, \dint u,
		\end{multline*}
		which is \cref{Bracket B}.
		For $k = 1$ (resp.\ $k = N+1$) the simplex
		$\Delta_0$ is a point and $L = e^{-uH}$ (resp.\ $R = e^{-(1-u)H}$),
		in accordance with the convention for empty brackets.
	\end{proof}
	Due to the analytic requirements (H2) and (H3) from \cref{Definition
	homotopy}, the operators  $\dot{\mathbf{c}}^s(\theta)$ and $\dot{Q}_s$
	satisfy the assumptions from \cite[Lemma 4.2]{GüneysuLudewig} so that
	$\dot{\omega}_{\mathscr{M}^s}\Phi_{(1-u)T}^{\mathscr{M}^s}$, and
	consequently the integrand in the definition of $\Psi^s_T$, is a
	well-defined $\Lopclifqtrace(\HH)$-valued bar cochain.
	Applying \cref{Bracket B} together with \cref{Product bar cochains,Defn
	Quantization}, it follows that $\Psi_T^s$ is well-defined
	and $\Psi_T^s(\theta_1, \ldots, \theta_N)$ is equal to
	\footnotesize
	\begin{align*}
		& \sum_{M = 1}^{N} \Bigg((-T)^{M+1} \sum_{j=1}^{M+1} \sum_{I \in
			\mathcal{P}_{M, N}} \{F_s^{\geq 1}(\theta_{I_1}),
		\ldots, F_s^{\geq 1}(\theta_{I_{j-1}}), \dot{Q}_s, F_s^{\geq
			1}(\theta_{I_j}), \ldots, F_s^{\geq 1}(\theta_{I_M})\}_{TQ_s^2} \\
		& \quad + (-T)^M\sum_{j=1}^{M} \sum_{k=1}^N \sum_{\substack{I \in
				\mathcal{P}_{M, N} \\
				I_j =  \{k\}}} \{F_s^{\geq 1}(\theta_{I_1}),
		\ldots, F_s^{\geq 1}(\theta_{I_{j-1}}), \dot{\mathbf{c}}^s(\theta_k),
		F_s^{\geq 1}(\theta_{I_{j+1}}), \ldots, F_s^{\geq
			1}(\theta_{I_M})\}_{TQ_s^2} \Bigg).
	\end{align*}
	\normalsize
	Using the uniform norm bound on $\dot{Q}_s \Delta_s^{-1/2}$, we may apply
	\cite[Lemma 4.2]{GüneysuLudewig} and repeat the steps from the proof of
	\cite[Theorem 4.6]{GüneysuLudewig} to estimate the two brackets in the
	above display.
	The inner sums give extra factors $N$ and $N^2$ respectively, which
	may be absorbed into the seminorm since $N^2 + N \leq 4^N$ whereas the
	factor $T^{(N+1)/2}$ is a result of the additional $T$ appearing in the
	definition of $\Psi^s_T$.
	Since $\nu$ may be chosen independently of $s$ due to the analytic
	requirements (H1) and (H3) from \cref{Definition homotopy}, we arrive at
	\cref{Estimate Psi}.
\end{proof}
Finally, denote by $\mathscr{M}^s_\T \coloneqq (\HH,\mathbf{c}_{\mathbb{T}}^s,
Q_s), s \in [0,1]$, the homotopy of $\vartheta$-summable
\textit{weak} $\mathcal{C}_q$-Fredholm modules over $\Omega_{\mathbb{T}}$
associated to the homotopy $\mathscr{M}^s$.
The \textit{Chern-Simons form} is then given by
\begin{equation*}
	\mathrm{CS}\big((\mathscr{M}_{\mathbb{T}}^s)_{s \in [0,1]}\big) \coloneqq
	(-1)^q\int_0^1 \alpha^* \CStr\big(
	\underline{\Psi}_1^{\mathscr{M}^s_{\mathbb{T}}}\big) \, \dint s
\end{equation*}
where $\alpha$ was defined in \cref{Definition alpha} and
$\underline{\Psi}_1^{\mathscr{M}^s_{\mathbb{T}}}$ is the quotient map of
$\Psi_1^{\mathscr{M}^s_{\mathbb{T}}}$ on
$\mathrm{B}(\underline{\Omega_{\mathbb{T}}})$.
\cref{Estimate Psi} and (H1) from \cref{Definition homotopy} show that this
is a well-defined analytic cochain.
Moreover, since $\Psi_1^{\mathscr{M}^s_{\mathbb{T}}}$ is odd, $\alpha$ is even
and $\CStr$ is even for $q$ even and odd for $q$ odd, it follows that the
Chern-Simons form is odd for $q$ even and even for $q$ odd.
Calculations similar to the ones for $\ChernT$ show that it vanishes on the
images of the operators in \cref{Chen operators} since the $\mathscr{M}^s$
satisfy \cref{Multiplicativity}.
Now, homotopy invariance is a consequence of
\begin{thm}
	The following transgression formula holds:
	\begin{equation*}
		\mathrm{Ch}_{\mathscr{M}^1_{\mathbb{T}}} -
		\mathrm{Ch}_{\mathscr{M}^0_{\mathbb{T}}} = (\underline{d}_{\mathbb{T}}
		+ \underline{b} -
		\underline{B})^{\vee}\mathrm{CS}\big((\mathscr{M}_{\mathbb{T}}^s)_{s
			\in [0,1]}\big).
	\end{equation*}
\end{thm}
\begin{proof}
	According to \cref{Ch pullback}, we have
	\begin{equation*}
		\mathrm{Ch}_{\mathscr{M}^s_{\mathbb{T}}} = -\alpha^*
		\CStr(\underline{\Phi}_1^{\mathscr{M}^s_{\mathbb{T}}}).
	\end{equation*}
	As in \cite[Proposition 6.4]{GüneysuLudewig}, one has
	\begin{equation}\label{Bianchi psi}
		\dfrac{d}{ds} \Phi_T^s = \delta \Psi_T^s + [\omega_s, \Psi_T^s].
	\end{equation}
	Thus,
	\begin{align*}
		\frac{d}{ds}\mathrm{Ch}_{\mathscr{M}^s_{\mathbb{T}}} & = -
		\CStr\Big(\frac{d}{ds}
		\underline{\Phi}_1^{\mathscr{M}_{\mathbb{T}}^s} \alpha\Big) = -
		\left(\CStr\big(\delta \underline{\Psi}_1^{\mathscr{M}_{\mathbb{T}}^s}
		\alpha\big) + \CStr\big([\omega^s,
		\underline{\Psi}_1^{\mathscr{M}_{\mathbb{T}}^s}]\alpha\big)\right) \\
		& = - \CStr
		\big(\underline{\Psi}_1^{\mathscr{M}_{\mathbb{T}}^s}(\underline{d}_\T
		+
		\underline{b}')\alpha\big),
	\end{align*}
	where for the last equality we used that the image of $\alpha$ is
	contained in
	$\mathrm{B}^{\sharp}(\underline{\Omega_{\mathbb{T}}})$ so that $[\omega^s,
	\underline{\Psi}_1^{\mathscr{M}_{\mathbb{T}}^s}]\alpha$
	is a sum of supercommutators on which $\CStr$ vanishes.
	Therefore, with a view on \cref{Interaction alpha}, we get
	\begin{equation*}
		\frac{d}{ds}\mathrm{Ch}_{\mathscr{M}^s_{\mathbb{T}}} = (-1)^{q}
		(\underline{d}_{\mathbb{T}} + \underline{b} - \underline{B})^{\vee}
		\alpha^*\CStr\big(\underline{\Psi}_1^{\mathscr{M}_{\mathbb{T}}^s}\big)
		- \CStr\big(\underline{\Psi}_1^{\mathscr{M}_{\mathbb{T}}^s}
		(\underline{S}' + \mathbf{1})h\big).
	\end{equation*}
	Note that the second summand is zero similarly to \cref{Clifford
	supertrace and S + id}.
	Hence, integrating both sides of the above equation yields the desired
	transgression formula.
\end{proof}
\begin{rem}
	Since the $\mathscr{M}^s$ satisfy \cref{Multiplicativity}, it follows
	that $\mathrm{Ch}_{\mathscr{M}^s_{\mathbb{T}}}$ is Chen normalized for
	all $s \in [0,1]$ due to Part 3 of \cref{Thm ChM}.
	Hence, the above formula shows a posteriori that
	$\mathrm{CS}\big((\mathscr{M}_{\mathbb{T}}^s)_{s \in [0,1]}\big)$
	vanishes on all of $\mathrm{D}^{\mathbb{T}}(\Omega)$ and is therefore
	Chen normalized as well.
\end{rem}
\subsection{Comparison to Getzler's odd Chern character}\label{Section: Getzler}
Let $\mathscr{M} = (\HH, \mathbf{c}, Q)$ be an odd $\vartheta$-summable weak
Fredholm module over a locally convex dg algebra $\Omega$.
We define its Chern character via
\begin{equation*}
	\Chern \coloneqq \mathrm{Ch}_{\tilde{\mathscr{M}}}
\end{equation*}
where $\tilde{\mathscr{M}}$ is the associated $\vartheta$-summable weak
$\mathcal{C}_1$-Fredholm module defined in \cref{Section: odd FM}.
Getzler's setting \cite{GetzlerOdd} corresponds to the case where $\Omega =
\mathcal{A}$ trivially graded with trivial differential and $\mathbf{c}$ is
multiplicative, i. e. a unital  representation of $\mathcal{A}$. (Keep in mind
that since $\mathscr{M}$ is only assumed to be weak, we do not require the
first relation in \cref{Equations Fredholm} to hold which would force $[Q,
\mathbf{c}(\theta)] = 0$.)
\begin{lem}\label{Lemma Getzler}
	In the above setting, $\Chern(\theta_0, \ldots, \theta_N) = 0$
	for $N$ even, and for $N$ odd
	\begin{multline}\label{Getzler formula}
		\Chern(\theta_0, \ldots, \theta_N) \\* = \int_{\Delta_N}
		\Tr\bigl(\mathbf{c}(\theta_0)e^{-\tau_1 Q^2}[Q, \mathbf{c}(\theta_1)]
		e^{-(\tau_2 - \tau_1) Q^2} \cdots [Q, \mathbf{c}(\theta_N)]
		e^{-(1 - \tau_N)Q^2}\bigr) \, \dint \tau.
	\end{multline}
	Up to signs and normalization constants, which result from different
	conventions for the supertrace and the cyclic complex, this coincides with
	the odd Chern character of \cite[Definition 4.1, Proposition
	4.2]{GetzlerOdd}.
\end{lem}
\begin{proof}
	Since $\mathbf{c}$ is a representation and $d$ is trivial, we have
	$F_{\tilde{\mathscr{M}}}(\theta_1, \theta_2) = 0$ and
	\begin{equation*}
		F_{\tilde{\mathscr{M}}}(\theta_1) = [\tilde{Q},
		\tilde{\mathbf{c}}(\theta_1)] - \tilde{\mathbf{c}}(d \theta_1)
		= [\tilde{Q},
		\tilde{\mathbf{c}}(\theta_1)]
		=
		\begin{pmatrix}
			0 & [Q, \mathbf{c}(\theta_1)]\\
			[Q, \mathbf{c}(\theta_1)] & 0
		\end{pmatrix}
	\end{equation*}
	for $\theta_1, \theta_2 \in \Omega$, which is why we have
	\begin{equation*}
		\Phi_1^{\tilde{\mathscr{M}}}(\theta_1, \ldots, \theta_N) = (-1)^N
		\{F_{\tilde{\mathscr{M}}}(\theta_1), \ldots,
		F_{\tilde{\mathscr{M}}}(\theta_N)\}_{\tilde{Q}^2}.
	\end{equation*}
	Using $e^{-s\tilde{Q}^2} =
	\mathrm{diag}(e^{-sQ^2}, e^{-sQ^2})$, each factor
	$e^{-s\tilde{Q}^2}F_{\tilde{\mathscr{M}}}(\theta_j)$ is off-diagonal with
	both off-diagonal entries equal to $e^{-sQ^2}[Q, \mathbf{c}(\theta_j)]$.
	Hence, $\Phi_1^{\tilde{\mathscr{M}}}(\theta_1, \ldots, \theta_N)$ equals
	$\mathrm{diag}(P, P)$ for $N$ even and
	$\begin{psmallmatrix} 0 & P \\ P & 0\end{psmallmatrix}$ for $N$ odd, where
	$P \coloneqq (-1)^N\{[Q, \mathbf{c}(\theta_1)], \ldots, [Q,
	\mathbf{c}(\theta_N)]\}_{Q^2}$.

	Since all elements of $\Omega$ are even,
	\begin{equation*}
		\gamma \tilde{\mathbf{c}}(\theta_0) =
		\begin{pmatrix}
			0 & -\mathbf{1} \\
			\mathbf{1} & 0
		\end{pmatrix}
		\cdot
		\begin{pmatrix}
			\mathbf{c}(\theta_0) & 0 \\
			0 & \mathbf{c}(\theta_0)
		\end{pmatrix}
		=
		\begin{pmatrix}
			0 & -\mathbf{c}(\theta_0) \\
			\mathbf{c}(\theta_0) & 0
		\end{pmatrix}.
	\end{equation*}
	Therefore, for $N$ even,
	$\gamma\tilde{\mathbf{c}}(\theta_0)\Phi_1^{\tilde{\mathscr{M}}}(\theta_1,
	\ldots, \theta_N)$ is off-diagonal and its supertrace vanishes which proves
	the first claim.
	Using the previous calculations, we see that for $N$ odd,
	$\gamma\tilde{\mathbf{c}}(\theta_0)\Phi_1^{\tilde{\mathscr{M}}}(\theta_1,
	\ldots, \theta_N) =
	\mathrm{diag}(-\mathbf{c}(\theta_0)P, \mathbf{c}(\theta_0)P)$, so that
	\begin{equation*}
		\Chern(\theta_0, \ldots, \theta_N) = \tfrac{1}{2}\Str\bigl(\gamma
		\tilde{\mathbf{c}}(\theta_0)\Phi_1^{\tilde{\mathscr{M}}}(\theta_1,
		\ldots, \theta_N)\bigr) = -\Tr\bigl(\mathbf{c}(\theta_0)P\bigr),
	\end{equation*}
	which is \cref{Getzler formula}, as $-(-1)^N = 1$ for $N$ odd.
\end{proof}
Thus, $\mathrm{Ch}_{\mathscr{M}}$ is a differential graded extension
of Getzler's odd Chern character.
\section{A noncommutative index theorem for the odd Chern
character}\label{Section: index}
Given an odd $\vartheta$-summable Fredholm module $\mathscr{M} = (\HH, \pi,
\mathrm{D})$ over a Banach $*$-algebra $\mathcal{A}$ in the sense of
\cite{ConnesCompact}, one can show that pairing the corresponding odd Chern
character $\mathrm{Ch}_*(\mathrm{D})$ with the odd Chern character
$\mathrm{Ch}^{*}(g)$ of a unitary matrix $g \in \mathrm{U}_m(\mathcal{A})$
yields the spectral flow of $\mathrm{D}$ and the twisted operator
$g^{-1}\mathrm{D}g$:
\begin{thm}[\protect{\cite[Theorem 4.3]{GetzlerOdd}}]\label{Getzler sf}
	We have
	\begin{equation}\label{Spectral flow pairing}
		\langle \mathrm{Ch}_*(\mathrm{D}), \mathrm{Ch}^{*}(g) \rangle =
		\frac{1}{\sqrt{\pi}} \int_0^1 \Tr(\dot{\mathrm{D}}_u
		e^{-\mathrm{D}_u^2}) \, \dint u = \mathrm{sf}(\mathrm{D},
		g^{-1}\mathrm{D}g),
	\end{equation}
	where $\mathrm{D}_u \coloneqq (1-u)\mathrm{D} + ug^{-1}\mathrm{D}g$ and
	$\dot{\mathrm{D}}_u \coloneqq g^{-1} [\mathrm{D}, g]$.
\end{thm}
Consider now a locally convex dg algebra $\Omega$, $g \in
\mathrm{GL}_m(\Omega^0)$ and an odd $\vartheta$-summable Fredholm
module $\mathscr{M} = (\HH, \mathbf{c}, Q)$ over $\Omega$ as introduced in
\cref{Definition odd FM}.
We want to construct an odd entire chain $\mathrm{Ch}(g)$ which resembles
the odd Chern character of $g$ in the differential graded setting.
Moreover, we would like to obtain an integral formula similar to
\cref{Getzler sf} for the pairing of the analytic cochain
$\mathrm{Ch}_{\tilde{\mathscr{M}}}$,
where $\tilde{\mathscr{M}}$ is the $\vartheta$-summable
$\mathcal{C}_1$-Fredholm module associated to $\mathscr{M}$, with
$\mathrm{Ch}(g)$.
\subsection{Definition of
	\texorpdfstring{$\mathrm{Ch}(g)$}{Ch(g)}}\label{Section: definition cherng}
We follow the construction in \cite[Section 5]{GüneysuCacciatori}, where $g$ is
a smooth map from a compact manifold to a unitary group, and make the change
necessary to accommodate an arbitrary locally convex dg algebra.
To this end, we will define an odd entire chain $\mathrm{Ch}(g) \in
\mathrm{C}^{\epsilon}_{-}(\Omega_{\mathbb{T}})$ on the acyclic extension
$\Omega_{\mathbb{T}}$ of $\Omega$.\footnote{Note that we follow the sign
	conventions from \cite[Eq. (2.7)]{GüneysuLudewig} for the definition of
	$\Omega_{\mathbb{T}}$, i. e. we
	consider $d_{\mathbb{T}} = d - \iota$ instead of $d + \iota$, the latter
	convention being used in \cite[Section 3]{GüneysuCacciatori}.}
\begin{defn}
	For $g \in \mathrm{GL}_m(\Omega^0)$ we define the \textit{Maurer-Cartan
		form} $\omega \coloneqq \omega_{g} \coloneqq g^{-1} dg \in
	\mathrm{Mat}_m(\Omega^1) \subseteq
	\mathrm{Mat}_m(\Omega_{\mathbb{T}})$.
\end{defn}
\begin{lem}
	The Maurer-Cartan form satisfies the following equation:
	\begin{equation*}
		d_{\mathbb{T}}\omega + \omega^2 = 0.
	\end{equation*}
\end{lem}
\begin{proof}
	Since $\omega \in \mathrm{Mat}_m(\Omega)$, the action of the
	differential $d_{\mathbb{T}}$ is given by $d$.
	Using that $g^{-1}$ has degree $0$, the graded Leibniz rule implies
	\begin{equation*}
		d \omega = d(g^{-1}dg) = (dg^{-1})dg + g^{-1}d^2 g = (dg^{-1})dg.
	\end{equation*}
	On the other hand,
	\begin{equation}\label{dg^-1}
		0 = d(\mathbf{1}) = d(g^{-1}g) = (dg^{-1})g + g^{-1}dg \iff dg^{-1} =
		-g^{-1}(dg)g^{-1}.
	\end{equation}
	Therefore,
	\begin{equation*}
		d \omega = (dg^{-1})dg = -g^{-1}(dg)g^{-1}dg = -(g^{-1}dg)(g^{-1}dg) =
		-\omega^2
	\end{equation*}
	as desired.
\end{proof}
Following \cite[Eq. (24)]{GüneysuCacciatori}, we define
\begin{equation*}
	\mathcal{A}^s(g) \coloneqq s \omega - s(1-s)\sigma \omega^2, \quad
	\mathcal{B}^s(g) \coloneqq \sigma \omega
\end{equation*}
in $\mathrm{Mat}_m(\Omega_{\mathbb{T}})$ for $s \in I \coloneqq
[0,1]$.
Note that we have slightly different signs in the above definitions since we
use the convention $d_{\mathbb{T}} = d - \iota$.
\begin{defn}
	The \textit{odd Chern character of $g$} is given by the infinite cyclic
	chain
	\begin{equation*}
		\mathrm{Ch}(g) \coloneqq (\mathrm{Ch}_0(g), \mathrm{Ch}_1(g), \ldots)
		\in \prod_{N = 0}^{\infty} \Omega_\T \otimes
		\underline{\Omega_\T}[1]^{\otimes N}
	\end{equation*}
	where $\mathrm{Ch}_0(g) \coloneqq 0$ and
	\begin{equation*}
		\mathrm{Ch}_N(g) \coloneqq \Tr \left(\int_0^1 \mathbf{1} \otimes
		\sum_{k=1}^{N} \mathcal{A}^s(g)^{\otimes(k-1)} \otimes
		\mathcal{B}^s(g)
		\otimes \mathcal{A}^s(g)^{\otimes(N-k)} \, \dint s\right)
	\end{equation*}
	for $N > 0$ and $\Tr$ denotes the generalized trace map,
	\begin{align*}\label{Definition generalized trace}
		\begin{split}
			\Tr \colon \mathrm{C}^\epsilon(\mathrm{Mat}_m(\Omega_\T)) &\to
			\mathrm{C}^\epsilon(\Omega_\T), \\
			(\theta_0, \ldots, \theta_N) &\mapsto \sum_{i_0, \ldots, i_N =
				1}^m
			\left((\theta_0)_{i_0}^{i_1}, (\theta_1)_{i_1}^{i_2}, \ldots,
			(\theta_N)_{i_N}^{i_0}\right).
		\end{split}
	\end{align*}
\end{defn}
Letting the differential operate entrywise, $\mathrm{Mat}_m(\Omega)$ becomes a
locally convex dg algebra, with acyclic extension given by
$\mathrm{Mat}_m(\Omega_\T)$, so that
$\mathrm{D}^\T(\mathrm{Mat}_m(\Omega)) \subseteq
\mathrm{C}(\mathrm{Mat}_m(\Omega_\T))$ is defined as in
\cref{Section: chen normalization}, with $f$ ranging over
$\mathrm{Mat}_m(\Omega^0)$.
The following lemma is the result of a simple computation.
\begin{lem}\label{Lemma trace Chen}
	The generalized trace map $\Tr$ commutes with $\underline{d}_\T$,
	$\underline{b}$ and $\underline{B}$, and it maps
	$\mathrm{D}^\T(\mathrm{Mat}_m(\Omega))$ into $\mathrm{D}^\T(\Omega)$ and
	$\overline{\mathrm{D}^\T(\mathrm{Mat}_m(\Omega))}$ into
	$\overline{\mathrm{D}^\T(\Omega)}$.
\end{lem}
We are now in a position to prove:
\begin{thm}\label{Thm Chg}
	The following assertions hold:
	\begin{enumerate}
		\item One has $\mathrm{Ch}(g) \in
		\mathrm{C}^{\epsilon}_{-}(\Omega_{\mathbb{T}})$.
		\item One has $(\underline{d}_{\mathbb{T}} + \underline{b} -
		\underline{B})\mathrm{Ch}(g) =
		0$ in the Chen normalized entire complex $\mathrm{N}^{\mathbb{T},
			\epsilon}(\Omega)$.
		\item If $(g_t)_{t \in I}$ is a smooth family of elements in
		$\mathrm{GL}_m(\Omega^0)$, then $\mathrm{Ch}(g_0)$ and
		$\mathrm{Ch}(g_1)$ are homologous in the Chen normalized entire
		complex $\mathrm{N}^{\mathbb{T},\epsilon}(\Omega)$.
	\end{enumerate}
\end{thm}
\begin{proof}
	1. Given a continuous seminorm $\nu$ on $\Omega_{\mathbb{T}}$, let
	\begin{equation*}
		C_{\nu} \coloneqq \sup_{s \in I} \, \max\left(\nu(\mathbf{1}),
		\max_{i,j =
			1, \ldots, m} \nu(\mathcal{A}^s(g)_{i}^j), \max_{i,j = 1, \ldots,
			m}
		\nu(\mathcal{B}^s(g)_{i}^j)\right).
	\end{equation*}
	With a view on \cref{Definition entire seminorm}, we see that
	\begin{equation}\label{Seminorm Ch(g)}
		\epsilon_{\nu}(\mathrm{Ch}(g)) \leq mC_{\nu}\sum_{N=0}^{\infty} N
		\dfrac{(m C_{\nu})^N}{\lfloor N/2 \rfloor!}.
	\end{equation}
	Since $\sqrt[N]{\lfloor N/2 \rfloor!} \to \infty$ as $N \to \infty$, it
	follows that
	\begin{equation*}
		\limsup_{N \to \infty}\sqrt[N]{\dfrac{N(mC_{\nu})^N}{\lfloor N/2
				\rfloor!}} = \limsup_{N \to \infty} \sqrt[N]{N} (m C_{\nu})
		\dfrac{1}{\sqrt[N]{\lfloor N/2 \rfloor!}} = 0
	\end{equation*}
	so that the sum in \cref{Seminorm Ch(g)} converges.
	Hence, $\mathrm{Ch}(g)$ is entire.
	It is easy to check that $\mathrm{Ch}(g)$ is odd.

	2. We have to prove that $(\underline{d}_{\mathbb{T}} + \underline{b} -
	\underline{B})\mathrm{Ch}(g) \in
	\overline{\mathrm{D}^{\mathbb{T}}(\Omega)}$.

	First of all, we notice that
	$\underline{B}\mathrm{Ch}(g)$ is a sum of elements of the form
	$S_i^{(\mathbf{1})}$ (see \cref{Si f}) so that $\underline{B}
	\mathrm{Ch}(g) \in \overline{\mathrm{D}^{\mathbb{T}}(\Omega)}$.

	Next, we prove that $(\underline{d}_{\mathbb{T}} +
	\underline{b})\mathrm{Ch}(g) \in
	\overline{\mathrm{D}^{\mathbb{T}}(\Omega)}$.
	As $\underline{d}_{\mathbb{T}}$ preserves and $\underline{b}$ reduces the
	length of chains by one, we have
	\begin{equation*}
		((\underline{d}_{\mathbb{T}} +
		\underline{b})\mathrm{Ch}(g))_N=((\underline{d}_{\mathbb{T}} +
		\underline{b})
		\mathrm{Ch}_N(g))_N+((\underline{d}_{\mathbb{T}} + \underline{b})
		\mathrm{Ch}_{N+1}(g))_N.
	\end{equation*}
	Let $\mathcal{A} \coloneqq \mathcal{A}^s(g)$, $\mathcal{B} \coloneqq
	\mathcal{B}^s(g)$ and $\theta^{(k)} \coloneqq \mathbf{1} \otimes
	\mathcal{A}^{\otimes (k-1)} \otimes \mathcal{B} \otimes
	\mathcal{A}^{\otimes (L-k)}$ for $1 \leq k \leq L$, so that
	$\mathrm{Ch}_L(g) = \Tr\int_0^1 \sum_{k=1}^L \theta^{(k)}\,\dint s$.
	The Maurer-Cartan equation and $d(\omega_g^2) = (d\omega_g)\omega_g -
	\omega_g\, d\omega_g = 0$ imply
	\begin{equation*}
		d_\T\mathcal{A} = -s^2\omega_g^2, \quad d_\T\mathcal{B} =
		\sigma\omega_g^2 - \omega_g, \quad \mathcal{A}\mathcal{A} =
		s^2\omega_g^2, \quad \mathcal{A}\mathcal{B} = -s\sigma\omega_g^2,
		\quad \mathcal{B}\mathcal{A} = s\sigma\omega_g^2.
	\end{equation*}
	Furthermore, since $\lvert \mathcal{A}^s(g) \rvert = 1$ and $\lvert
	\mathcal{B}^s(g) \rvert = 0$ we have $m_{i-1}(\theta^{(k)}) \equiv 0$ for
	$i \leq k$ and
	$m_{i-1}(\theta^{(k)}) \equiv 1$ for $i > k$ modulo $2$.
	Plugging this into definitions of $\underline{d}_{\mathbb{T}}$ and
	$\underline{b}$, a closer inspection reveals that
	$(\underline{b}\,\mathrm{Ch}_{N+1}(g))_N$ has terms of three kinds:
	\begin{enumerate}[label=(\roman*)]
		\item the terms multiplying $\theta_0 = \mathbf{1}$ with its right
		neighbour, and the cyclic terms multiplying $\theta_{N+1}$ with
		$\theta_0$; since $(-1)^{(\lvert\theta^{(k)}_{N+1}\rvert - 1)m_N} = 1$
		for all $k$ and the cyclic rotation permutes the summands
		$\theta^{(k)}$, these terms cancel each other;
		\item the terms multiplying two adjacent $\mathcal{A}$; they
		cancel against the terms of $(\underline{d}_\T
		\mathrm{Ch}_N(g))_N$ in which $d_\T$ acts on $\mathcal{A}$;
		\item the terms multiplying $\mathcal{B}$ with a neighbouring
		$\mathcal{A}$, which together contribute $2s\sigma\omega_g^2$ in the
		position of $\mathcal{B}$.
	\end{enumerate}
	The remaining terms of $(\underline{d}_\T\mathrm{Ch}_N(g))_N$, in which
	$d_\T$ acts on $\mathcal{B}$, contribute $\omega_g - \sigma\omega_g^2$ in
	the position of $\mathcal{B}$.
	Using $\frac{d}{ds}\mathcal{A}^s(g) =
	\omega_g + (2s-1)\sigma\omega_g^2$, $\mathcal{A}^1(g) = \omega_g$ and
	$\mathcal{A}^0(g) = 0$, we obtain
	\begin{align*}
		((\underline{d}_\T + \underline{b})\mathrm{Ch}(g))_N & = \Tr
		\left(\int_0^1 \sum_{k = 1}^N \mathbf{1} \otimes
		\mathcal{A}^s(g)^{\otimes(k-1)} \otimes \left(\frac{d}{ds}
		\mathcal{A}^s(g)\right) \otimes \mathcal{A}^s(g)^{\otimes (N-k)} \,
		\dint s\right) \\
		& = \Tr\left(\int_0^1 \frac{d}{ds}(\mathbf{1} \otimes
		\mathcal{A}^s(g)^{\otimes N}) \, \dint s\right) \\
		& = \Tr\big(\mathbf{1} \otimes \mathcal{A}^1(g)^{\otimes N}\big) -
		\Tr\big(\mathbf{1} \otimes \mathcal{A}^0(g)^{\otimes N}\big) \\
		& = \Tr(\mathbf{1} \otimes \omega_g^{\otimes N}).
	\end{align*}
	\begin{claim}
		Define\footnote{The last part of the proof in
			\cite{GüneysuCacciatori} implicitly uses that the tensor product
			is taken over $\Omega^0$ which is wrong.
			The following correction was communicated to me by Sergio
			Cacciatori.}
		\begin{align*}
			\lambda_2 & \coloneqq \Tr(\mathbf{1} \otimes dg^{-1} \otimes dg)
			\\
			\lambda_N & \coloneqq \Tr(\mathbf{1} \otimes
			\omega_g^{\otimes(N-2)}
			\otimes dg^{-1} \otimes dg), \quad N \geq 3
		\end{align*}
		in $\mathrm{C}(\Omega_{\mathbb{T}})$.
		Then,
		\begin{align*}
			\Tr(\mathbf{1} \otimes \omega_g) - \lambda_2 & \in
			\mathrm{D}^{\mathbb{T}}(\Omega)\\
			\Tr(\mathbf{1} \otimes \omega_g^{\otimes N}) + \lambda_N -
			\lambda_{N+1} & \in
			\mathrm{D}^{\mathbb{T}}(\Omega), \quad N \geq 2.
		\end{align*}
	\end{claim}
	\begin{proof}[Proof of Claim]
		For $N = 1$, we have
		\begin{equation*}
			\mathbf{1} \otimes \omega_g = \mathbf{1} \otimes g^{-1}dg =
			g^{-1} \otimes dg + \mathbf{1} \otimes dg^{-1} \otimes dg
		\end{equation*}
		modulo $\mathrm{D}^{\mathbb{T}}(\mathrm{Mat}_m(\Omega))$.
		Here we used $T_0^{(g^{-1})}(\mathbf{1}, dg) = (\mathbf{1},
		\omega_g) - (g^{-1}, dg) - (\mathbf{1}, dg^{-1}, dg)$.
		However, since $g^{-1} \otimes dg = -T_0^{(g)}(g^{-1})$ (see
		\cref{Section: chen normalization})
		it follows that
		\begin{equation*}
			\mathbf{1} \otimes \omega_g - \mathbf{1} \otimes dg^{-1} \otimes
			dg \in \mathrm{D}^{\mathbb{T}}(\mathrm{Mat}_m(\Omega))
		\end{equation*}
		which implies $\Tr(\mathbf{1} \otimes \omega_g) - \lambda_2 \in
		\mathrm{D}^{\mathbb{T}}(\Omega)$ by \cref{Lemma trace Chen}.

		For $N \geq 2$, we have
		\begin{align*}
			\mathbf{1} \otimes \omega_g^{\otimes N} & = \mathbf{1} \otimes
			\omega_g^{\otimes (N-2)} \otimes g^{-1}dg \otimes g^{-1} dg \\
			& = \mathbf{1} \otimes \omega_g^{\otimes (N-2)} \otimes
			g^{-1}(dg)g^{-1} \otimes dg + \mathbf{1} \otimes
			\omega_g^{\otimes(N-1)} \otimes dg^{-1} \otimes dg \\
			& = -\mathbf{1} \otimes
			\omega_g^{\otimes (N-2)} \otimes dg^{-1} \otimes dg + \mathbf{1}
			\otimes \omega_g^{\otimes(N-1)} \otimes dg^{-1} \otimes dg
		\end{align*}
		modulo $\mathrm{D}^{\mathbb{T}}(\mathrm{Mat}_m(\Omega))$, where we
		applied $T_{N-1}^{(g^{-1})}$ to $(\mathbf{1}, \omega_g^{\otimes
		(N-1)}, dg)$ and used \cref{dg^-1}.
		Hence, again by \cref{Lemma trace Chen},
		\begin{equation*}
			\Tr(\mathbf{1} \otimes \omega_g^{\otimes N}) + \lambda_N -
			\lambda_{N+1} \in
			\mathrm{D}^{\mathbb{T}}(\Omega)
		\end{equation*}
		as desired.
	\end{proof}
	Similarly to \cref{Seminorm Ch(g)}, one sees that the $\lambda_N$ form an
	entire chain so that in particular, the sequence $(\lambda_N)_N$
	converges to $0$ in
	$\mathrm{C}^\epsilon(\Omega_{\mathbb{T}})$ as $N \to \infty$.
	Hence, we may write
	\begin{equation*}
		0 = - \lambda_2 + \sum_{N = 2}^\infty (\lambda_N - \lambda_{N+1}).
	\end{equation*}
	As $((\underline{d}_\T + \underline{b})\mathrm{Ch}(g))_0 = 0$, it follows
	that
	\begin{equation*}
		(\underline{d}_\T + \underline{b})\mathrm{Ch}(g) =
		\bigl(\Tr(\mathbf{1}\otimes\omega_g) - \lambda_2\bigr) +
		\sum_{N=2}^\infty \bigl(\Tr(\mathbf{1}\otimes\omega_g^{\otimes N}) +
		\lambda_N - \lambda_{N+1}\bigr)
	\end{equation*}
	converges in $\mathrm{C}^\epsilon(\Omega_\T)$ with all terms lying
	in $\mathrm{D}^\T(\Omega)$ by the claim.
	Thus, $(\underline{d}_\T +
	\underline{b})\mathrm{Ch}(g) \in
	\overline{\mathrm{D}^{\mathbb{T}}(\Omega)}$.

	3. We construct an entire chain $\mathrm{CS}((g_t)_{t \in I})
	\in \mathrm{C}^\epsilon(\Omega_{\mathbb{T}})$ such that the transgression
	formula
	\begin{equation*}
		\mathrm{Ch}(g_1) - \mathrm{Ch}(g_0) = (\underline{d}_{\mathbb{T}} +
		\underline{b} - \underline{B})\mathrm{CS}((g_t)_{t \in I})
	\end{equation*}
	holds true in $\mathrm{N}^{\mathbb{T}, \epsilon}(\Omega)$.
	To this end, let
	\begin{equation*}
		\Omega[I] \coloneqq \Omega(I) \, \hat{\otimes} \, \Omega
	\end{equation*}
	where $\Omega(I)$ is the locally convex differential graded algebra of
	smooth differential forms on the interval $I$ and $\hat{\otimes}$ denotes
	the completed projective tensor product.
	Generic elements in $\Omega[I]$ will be denoted by $\vartheta$ whereas
	elements in $\Omega$ will be denoted by $\theta$.
	A $\Z$-grading on $\Omega[I]$ is given by
	\begin{equation*}
		\Omega[I]^n \coloneqq \bigoplus_{k + l = n} \Omega^k(I)  \,
		\hat{\otimes} \, \Omega^l
	\end{equation*}
	for $n \in \Z$.
	We define a product on $\Omega[I]$ by setting
	\begin{equation*}
		(\alpha_1 \otimes \theta_1) \cdot (\alpha_2 \otimes \theta_2) \coloneqq
		(-1)^{\lvert \theta_1 \rvert \lvert \alpha_2 \rvert}(\alpha_1 \alpha_2)
		\otimes (\theta_1 \theta_2)
	\end{equation*}
	for homogeneous $\alpha_i \in \Omega(I)$ and $\theta_i \in \Omega$ and
	extending continuously.
	Moreover, we define a differential $d$ on $\Omega[I]$ via
	\begin{equation*}
		d(\alpha \otimes \theta) \coloneqq (d \alpha) \otimes \theta +
		(-1)^{\lvert \alpha \rvert} \alpha \otimes d\theta
	\end{equation*}
	for homogeneous $\alpha \in \Omega(I)$ and $\theta \in \Omega$.
	Clearly, these two operations turn $\Omega[I]$ into a locally convex
	differential graded algebra.
	Note that for the acyclic extension we have an isomorphism $\Omega[I]_\T
	\simeq \Omega_\T[I]$ given by
	\begin{equation*}
		\alpha_1 \otimes
		\theta_1 + \sigma (\alpha_2 \otimes \theta_2) \mapsto
		\alpha_1 \otimes
		\theta_1 + (-1)^{\lvert \alpha_2 \rvert} \alpha_2 \otimes \sigma
		\theta_2
	\end{equation*}
	for homogeneous $\alpha_i \in \Omega(I)$ and $\theta_i \in \Omega$.
	Define continuous linear maps
	\begin{align*}
		j_t^* \colon \Omega_\T[I] \to \Omega_\T, \quad (g + f \dint t)
		\otimes
		\theta & \mapsto g(t) \theta, \\
		\iota_{\partial_I} \colon \Omega_\T[I] \to \Omega_\T[I], \quad (g + f
		\dint t) \otimes \theta & \mapsto f \otimes \theta
	\end{align*}
	for $t \in I$ and note that $j_t^*$ is degree-preserving while
	$\iota_{\partial_I}$ reduces degrees by one.
	They induce continuous linear maps
	\begin{align*}
		\underline{j}_t^* \colon \mathrm{C}(\Omega_\T[I]) \to
		\mathrm{C}(\Omega_\T),
		\quad (\vartheta_0, \ldots, \vartheta_N) & \mapsto (j_t^* \vartheta_0,
		\ldots, j_t^*\vartheta_N) \\
		\underline{\iota}_{\partial_I} \colon \mathrm{C}(\Omega_\T[I]) \to
		\mathrm{C}(\Omega_\T[I]), \quad (\vartheta_0, \ldots, \vartheta_N) &
		\mapsto - \sum_{k=0}^N (-1)^{m_{k-1}} (\vartheta_0, \ldots,
		\iota_{\partial_I} \vartheta_k, \ldots, \vartheta_N)
	\end{align*}
	which extend to the respective entire complexes.
	A standard calculation shows that
	\begin{equation}\label{Identites jiota}
		\begin{alignedat}{3}
			\underline{j}_t^* \circ \underline{d}_\T &= \underline{d}_\T \circ
			\underline{j}_t^*, & \underline{j}_t^* \circ \underline{b} &=
			\underline{b} \circ \underline{j}_t^*, & \underline{j}_t^* \circ
			\underline{B} &= \underline{B} \circ \underline{j}_t^*, \\
			\underline{\iota}_{\partial_I} \circ \underline{d}_\T +
			\underline{d}_\T \circ
			\underline{\iota}_{\partial_I} & = \frac{\partial}{\partial t}, &
			\underline{\iota}_{\partial_I} \circ \underline{b} +
			\underline{b} \circ \underline{\iota}_{\partial_I} & = 0, &
			\underline{\iota}_{\partial_I} \circ
			\underline{B} + \underline{B} \circ
			\underline{\iota}_{\partial_I} & = 0
		\end{alignedat}
	\end{equation}
	where
	\begin{equation*}
		\frac{\partial}{\partial t} (\vartheta_0, \ldots, \vartheta_N) =
		\sum_{k = 0}^n \left(\vartheta_0, \ldots, \frac{\partial}{\partial t}
		\vartheta_k, \ldots, \vartheta_N\right)
	\end{equation*}
	denotes the time derivative on $\mathrm{C}(\Omega_\T[I])$.\footnote{For
	$\vartheta = (g + f \dint t) \otimes \theta$ we have
	$\frac{\partial}{\partial t} \vartheta = (g' + f' \dint t) \otimes
	\theta$.}
	We may view $(g_t)_{t \in I}$ as an element
	\begin{equation*}
		g. \in \Omega^0(I) \, \hat{\otimes} \, \mathrm{Mat}_m(\Omega^0) \simeq
		\mathrm{Mat}_m(\Omega^0(I) \, \hat{\otimes} \, \Omega^0) \subseteq
		\mathrm{Mat}_m(\Omega[I]),
	\end{equation*}
	and since $g_t
	\in \mathrm{GL}_m(\Omega^0)$ for all $t \in I$, we have $g. \in
	\mathrm{GL}_m(\Omega^0(I) \, \hat{\otimes} \, \Omega^0) \subseteq
	\mathrm{GL}_m(\Omega[I]^0)$.
	Therefore, in the spirit of \cite{Simons}, we may define the
	\textit{Chern-Simons form}
	\begin{equation*}
		\mathrm{CS}((g_t)_{t \in I}) \coloneqq \int_0^1 \underline{j}_t^*
		\underline{\iota}_{\partial_I} \mathrm{Ch}(g.) \, \dint t \in
		\mathrm{C}^\epsilon(\Omega_\T).
	\end{equation*}
	Now, \cref{Identites jiota} yields
	\begin{equation*}
		(\underline{d}_\T + \underline{b} - \underline{B}) \mathrm{CS}((g_t)_{t
		\in I}) = -\int_0^1 \underline{j}_t^*
		\underline{\iota}_{\partial_I} (\underline{d}_\T + \underline{b} -
		\underline{B})\mathrm{Ch}(g.) \, \dint t + \int_0^1
		\underline{j}_t^* \frac{\partial}{\partial t}
		\mathrm{Ch}(g.) \, \dint t.
	\end{equation*}
	Since $\underline{j}_t^* \frac{\partial}{\partial t} \mathrm{Ch}(g.) =
	\frac{\partial}{\partial t} \underline{j}_t^* \mathrm{Ch}(g.) =
	\frac{\partial}{\partial t} \mathrm{Ch}(g_t)$, this implies
	\begin{equation*}
		(\underline{d}_\T + \underline{b} - \underline{B}) \mathrm{CS}((g_t)_{t
		\in I}) = -
		\int_0^1 \underline{j}_t^*
		\underline{\iota}_{\partial_I} (\underline{d}_\T + \underline{b} -
		\underline{B})\mathrm{Ch}(g.) \, \dint t + \mathrm{Ch}(g_1) -
		\mathrm{Ch}(g_0).
	\end{equation*}
	We have already shown that $(\underline{d}_\T + \underline{b} -
	\underline{B})\mathrm{Ch}(g.) = 0$ in $\mathrm{N}^{\mathbb{T},
	\epsilon}(\Omega[I])$.
	In fact, a closer examination of the proof reveals that
	\begin{equation*}
		(\underline{d}_\T + \underline{b} -
		\underline{B})\mathrm{Ch}(g.) \in
		\overline{\tilde{\mathrm{D}}^\T(\Omega[I])}
	\end{equation*}
	where $\tilde{\mathrm{D}}^\T(\Omega[I])$ is the \textit{subspace} of
	$\mathrm{D}^\T(\Omega[I])$ generated by the images of the operators
	$S_i^{(\vartheta)}$ and $T_i^{(\vartheta)}$
	for all $\vartheta \in \Omega^0(I) \, \hat{\otimes} \, \Omega^0$ and $i \in
	\N_0$, see \cref{Section: chen normalization} for the respective
	definitions.
	Thus, it remains to be shown that $\underline{j}_t^*
	\underline{\iota}_{\partial_I}$ maps
	$\tilde{\mathrm{D}}^\T(\Omega[I])$ into $\mathrm{D}^\T(\Omega)$.
	To this end, we calculate for all $g_1, g_2, f_2 \in \Omega^0(I), \theta_1
	\in \Omega^0, \theta_2 \in \Omega$:
	\begin{align*}
		\iota_{\partial_I}(g_1 \otimes \theta_1) & = 0, \\
		j_t^*(d(g_1 \otimes \theta_1)) & = g_1(t) d\theta_1 = d j_t^*(g_1
		\otimes \theta_1), \\
		j_t^* \iota_{\partial_I}(d(g_1 \otimes \theta_1)) & = g_1'(t) \theta_1
		= j_t^* \frac{\partial}{\partial t} (g_1 \otimes \theta_1),
		\\
		j_t^*\big((g_1 \otimes \theta_1)((g_2 + f_2
		\dint t) \otimes \theta_2)\big) & = (g_1 g_2)(t) \theta_1 \theta_2
		\\
		& = j_t^*(g_1 \otimes \theta_1) j_t^*((g_2 + f_2 \dint t) \otimes
		\theta_2), \\
		j_t^* \iota_{\partial_I} \big((g_1 \otimes \theta_1)((g_2 + f_2
		\dint t) \otimes \theta_2)\big) & = (g_1 f_2)(t) \theta_1 \theta_2
		\\
		& = j_t^*(g_1 \otimes \theta_1) j_t^* \iota_{\partial_I}((g_2 + f_2
		\dint t) \otimes \theta_2).
	\end{align*}
	It follows that $\underline{j}_t^*
	\underline{\iota}_{\partial_I}S_i^{(\vartheta)} = -S_i^{(j_t^* \vartheta)}
	\underline{j}_t^* \underline{\iota}_{\partial_I}$ and
	\begin{equation*}
		\underline{j}_t^* \underline{\iota}_{\partial_I}T_i^{{(\vartheta)}} =
		T_i^{(j_t^* \vartheta)} \underline{j}_t^*
		\underline{\iota}_{\partial_I} + S_i^{(j_t^* \frac{\partial}{\partial
		t}\vartheta)} \underline{j}_t^*
	\end{equation*}
	for all $\vartheta \in \Omega^0(I) \, \hat{\otimes} \, \Omega^0$ which
	proves
	\begin{equation*}
		\underline{j}_t^*
		\underline{\iota}_{\partial_I}\big(\tilde{\mathrm{D}}^\T
		(\Omega[I])\big)\subseteq \mathrm{D}^\T(\Omega)
	\end{equation*}
	as desired.
\end{proof}
\subsection{The spectral flow of an odd Fredholm module}\label{Section:
spectral flow}
We want to make sense of the right-hand side of \cref{Spectral flow pairing}
in our abstract setting.
To this end, let $\mathscr{M} = (\HH, \mathbf{c}, Q)$ be an odd
$\vartheta$-summable Fredholm module over the locally convex dg algebra
$\Omega$ and $\mathscr{M}_{\mathbb{T}}$ its acyclic extension as discussed in
\cref{Acyclic extension FM}.
For $m \in \N$ let $Q_m \coloneqq \mathrm{diag}(Q, \ldots, Q)$ on
$\HH^m = \HH \otimes \C^m$ and define $\mathbf{c}_{\T, m} \colon
\mathrm{Mat}_m(\Omega_\T) \to \Lop(\HH^m)$ entrywise by
$\mathbf{c}_{\T, m}(\theta) \coloneqq (\mathbf{c}_\T(\theta_i^j))_{i,j}$.
This gives an odd $\vartheta$-summable weak Fredholm module
$\mathscr{M}_{\mathbb{T}}^{(m)} = (\HH^m, \mathbf{c}_{\mathbb{T}, m},
Q_m)$ over the locally convex dg algebra
$\mathrm{Mat}_m(\Omega_{\mathbb{T}})$.
Recall that if no grading is specified, we assume all elements to have
degree $0$ so that in particular, the supercommutator $[\cdot, \cdot]$
reduces to the ordinary
commutator whereas the notation $[[\cdot, \cdot]]$ will always be used for
anticommutators.
Define unbounded operators
\begin{align*}
	Q_{g,s} & \coloneqq (1-s)Q_m + s \cdot \mathbf{c}_{\T, m}(g^{-1}) Q_m
	\mathbf{c}_{\T, m}(g) = Q_m + s \cdot \mathbf{c}_{\T, m}(g^{-1})[Q_m,
	\mathbf{c}_{\T, m}(g)]\\
	\dot{Q}_{g,s} & \coloneqq \mathbf{c}_{\T, m}(g^{-1})[Q_m, \mathbf{c}_{\T,
		m}(g)]
\end{align*}
on $\HH^m$ for $s \in I$.
Since $g, g^{-1} \in \mathrm{Mat}_m(\Omega^0)$ it follows
that $[Q_m, \mathbf{c}_{\T, m}(g)] = \mathbf{c}_{\T, m}(d g)$ and therefore,
\begin{equation*}
	\mathbf{c}_{\T, m}(g^{-1})[Q_m, \mathbf{c}_{\T, m}(g)] = \mathbf{c}_{\T,
		m}(g^{-1}) \mathbf{c}_{\T, m}(dg) =
	\mathbf{c}_{\T, m}(g^{-1}dg) = \mathbf{c}_{\T, m}(\omega).
\end{equation*}
Hence,
\begin{equation*}
	Q_{g,s} = Q_m + s \cdot \mathbf{c}_{\T, m}(\omega), \quad \dot{Q}_{g,s} =
	\mathbf{c}_{\T, m}(\omega)
\end{equation*}
and
\begin{equation*}
	Q_{g,s}^2 = Q_m^2 + s[[Q_m, \mathbf{c}_{\T, m}(\omega)]] + s^2
	\mathbf{c}_{\T, m}(\omega)^2 = Q_m^2 + X_s,
\end{equation*}
where $X_s \coloneqq s[[Q_m, \mathbf{c}_{\T, m}(\omega)]] + s^2
\mathbf{c}_{\T, m}(\omega)^2$.
\begin{lem}
	The operator $Q_{g,s}^2$ generates a strongly continuous semigroup of
	operators on $\HH^m$ with infinitesimal generator $-Q_{g,s}^2$ which is
	given by the perturbation series
	\begin{equation}\label{Perturbation series}
		e^{-TQ_{g,s}^2} = e^{-T(Q_m^2 + X_s)} = \sum_{M=0}^{\infty} (-T)^M
		\{\underbrace{X_s, \ldots, X_s}_M\}_{TQ_m^2}
	\end{equation}
	for $T > 0$.
	Moreover, the operator $\dot{Q}_{g,s} e^{-TQ_{g,s}^2}$ is well-defined
	and trace-class.
\end{lem}
\begin{proof}
	Due to the analytic requirement (A1) in \cref{Definition odd FM}, the
	operators $X_s$ satisfy
	the assumptions from \cite[Lemma 4.2]{GüneysuLudewig} with $a_j = \frac
	12$ for $1 \leq j
	\leq M$ and the same is true for $\dot{Q}_{g,s} =
	\mathbf{c}_{\T,m}(\omega)$
	with $a_0 = 0$ since $\mathbf{c}_{\T, m}$ maps to bounded operators.
	Thus,
	\begin{equation*}
		\dot{Q}_{g,s}\{\underbrace{X_s, \ldots, X_s}_{M}\}_{TQ_m^2}
	\end{equation*}
	is a well-defined trace-class operator for each $T > 0$ and the estimates
	from \cite[Lemma 4.2]{GüneysuLudewig} show that the sum in
	\cref{Perturbation series}
	converges in trace norm and defines a strongly continuous semigroup with
	infinitesimal generator $-Q_{g,s}^2$.\footnote{See also
		\cite[Chapter Nine, §2.1]{Kato}.}
\end{proof}
Thus, the integral
\begin{equation*}
	\int_0^1 \Tr(\dot{Q}_{g,s}e^{-Q_{g,s}^2}) \, \dint s
\end{equation*}
is well-defined when using the perturbation series for $e^{-Q_{g,s}^2}$.
Note that if $\mathbf{c}_{\T,m}(g)$ is unitary, then $Q_{g,1} = \mathbf{c}_{\T,m}(g)^{-1} Q_m
\mathbf{c}_{\T,m}(g)$ is unitarily equivalent to $Q_m$, $\dot{Q}_{g,s}$ is bounded,
and the above integral
equals $\sqrt{\pi} \, \mathrm{sf}(Q_m, \mathbf{c}_{\T,m}(g)^{-1}Q_m \mathbf{c}_{\T,m}(g))$
by \cite[Corollary 2.7]{GetzlerOdd}.
\subsection{The pairing \texorpdfstring{$\langle \ChernT,
		\mathrm{Ch}(g) \rangle$}{<ChMT, Ch(g)>}}\label{Section: spectral flow
	pairing 1}\label{Identifications of modules and pairing}
Recall that
\begin{align*}
	\mathrm{Ch}_N(g) & \coloneqq \Tr \left(\int_0^1 \mathbf{1} \otimes
	\sum_{k=1}^{N} \mathcal{A}^s(g)^{\otimes(k-1)} \otimes
	\mathcal{B}^s(g)
	\otimes \mathcal{A}^s(g)^{\otimes(N-k)} \, \dint s\right)\\
	& = \int_0^1 \Tr \left(\mathbf{1} \otimes
	\sum_{k=1}^{N} \mathcal{A}^s(g)^{\otimes(k-1)} \otimes
	\mathcal{B}^s(g)
	\otimes \mathcal{A}^s(g)^{\otimes(N-k)}\right)\, \dint s \\
	& = \int_0^1 \mathbf{1} \otimes \Tr \left(
	\sum_{k=1}^{N} \mathcal{A}^s(g)^{\otimes(k-1)} \otimes
	\mathcal{B}^s(g)
	\otimes \mathcal{A}^s(g)^{\otimes(N-k)}\right) \, \dint s
\end{align*}
for $N \geq 1$.
\cref{Doubling acyclic} shows that
$\tilde{\mathscr{M}}_{\mathbb{T}} = \widetilde{\mathscr{M}_{\mathbb{T}}}$,
and it is easy to check that
\begin{equation}\label{Iso U}
	U \colon \widetilde{\HH^m} \to \tilde{\HH}^m, \quad (v^{+}, v^{-})
	\mapsto \big((v^{+}_1, v^{-}_{1}), \ldots, (v^{+}_m, v^{-}_m)\big)
\end{equation}
induces an isomorphism between $\widetilde{\mathscr{M}^{(m)}_{\mathbb{T}}} =
(\widetilde{\HH^m}, \widetilde{\mathbf{c}_{\T, m}}, \widetilde{Q_m})$ and
$\tilde{\mathscr{M}}_{\mathbb{T}}^{(m)} = (\tilde{\HH}^m,
\tilde{\mathbf{c}}_{\T, m}, \tilde{Q}_m)$ as $\vartheta$-summable weak
$\mathcal{C}_1$-Fredholm modules over $\mathrm{Mat}_m(\Omega_{\mathbb{T}})$.
Since their Chern characters
coincide according to \cref{Isomorphism Chern character}, we will identify
the two Fredholm modules from now on.
Hence, using the identity
\begin{equation}\label{Str clif1 tr}
	\mathrm{Str}_{\mathcal{C}_1, \tilde{\HH}}\left((\tilde{\mathbf{c}}_\T
	\Phi_1^{\tilde{\mathscr{M}}_{\mathbb{T}}})(\Tr(\theta_0, \ldots,
	\theta_N))\right) = \mathrm{Str}_{\mathcal{C}_1,
		\tilde{\HH}^m}\left(\tilde{\mathbf{c}}_{\T, m}(\theta_0)
	\Phi_1^{\tilde{\mathscr{M}}_{\mathbb{T}}^{(m)}}(\theta_1, \ldots,
	\theta_N)\right),
\end{equation}
we have
\begin{align}
	&\langle \mathrm{Ch}_{\mathscr{M}_{\mathbb{T}}},
	\mathrm{Ch}(g) \rangle\notag \\
	& \qquad = \sum_{N=0}^{\infty} \int_0^1
	\mathrm{Str}_{\mathcal{C}_1,
		\tilde{\HH}^m}\left(\tilde{\mathbf{c}}_{\mathbb{T}, m}(\mathbf{1})
	\Phi_1^{\tilde{\mathscr{M}}_{\mathbb{T}}^{(m)}}\left(
	\sum_{k=1}^{N} \mathcal{A}^s(g)^{\otimes(k-1)} \otimes
	\mathcal{B}^s(g)
	\otimes \mathcal{A}^s(g)^{\otimes(N-k)}
	\right)\right) \notag\\
	& \qquad = \sum_{N=0}^{\infty} \int_0^1
	\mathrm{Str}_{\mathcal{C}_1, \tilde{\HH}^m}\left(
	\sum_{k=1}^{N}\Phi_1^{\tilde{\mathscr{M}}_{\mathbb{T}}^{(m)}}\left(
	\mathcal{A}^s(g)^{\otimes(k-1)} \otimes
	\mathcal{B}^s(g)
	\otimes \mathcal{A}^s(g)^{\otimes(N-k)}
	\right)\right) \notag\\
	& \qquad = -\sum_{N=0}^{\infty} \int_0^1
	\frac 12 \mathrm{Tr}_{\tilde{\HH}^m}\left(\tilde{\gamma}
	\sum_{k=1}^{N}\Phi_1^{\tilde{\mathscr{M}}_{\mathbb{T}}^{(m)}}\left(
	\mathcal{A}^s(g)^{\otimes(k-1)} \otimes
	\mathcal{B}^s(g)
	\otimes \mathcal{A}^s(g)^{\otimes(N-k)}
	\right)\right)\label{Pairing}
\end{align}
where
\begin{equation*}
	\tilde{\gamma} \coloneqq \begin{pmatrix}
		0 & \mathbf{1}_{\HH} \\
		\mathbf{1}_{\HH} & 0 \\
		& & \ddots \\
		& & & 0 & \mathbf{1}_{\HH} \\
		& & & \mathbf{1}_{\HH} & 0\\
	\end{pmatrix} \in \Lop(\tilde{\HH}^m).
\end{equation*}
To proceed further, we have to calculate the expressions involving
$\Phi_1^{\tilde{\mathscr{M}}_{\mathbb{T}}^{(m)}}$:
\begin{align*}
	& \Phi_1^{\tilde{\mathscr{M}}_{\mathbb{T}}^{(m)}}
	\left(\mathcal{A}^s(g)^{\otimes(k-1)} \otimes
	\mathcal{B}^s(g)
	\otimes \mathcal{A}^s(g)^{\otimes(N-k)}\right) \\
	& \qquad = \sum_{M = 1}^N (-1)^M \sum_{J \in \mathcal{P}_{M,N}}
	\{F_{\tilde{\mathscr{M}}_{\mathbb{T}}^{(m)}}(\theta_{J_1}),
	\ldots,
	F_{\tilde{\mathscr{M}}_{\mathbb{T}}^{(m)}}(\theta_{J_M})\}_
	{\tilde{Q}_m^2}
\end{align*}
where $\theta_1 \otimes \ldots \otimes \theta_N \coloneqq
\mathcal{A}^s(g)^{\otimes(k-1)} \otimes \mathcal{B}^s(g)
\otimes \mathcal{A}^s(g)^{\otimes(N-k)}$.
By definition of the bracket, we have to evaluate the following expressions:
\begin{align*}
	& F_{\tilde{\mathscr{M}}_{\mathbb{T}}^{(m)}}(\mathcal{A}^s(g)), \quad
	F_{\tilde{\mathscr{M}}_{\mathbb{T}}^{(m)}}(\mathcal{A}^s(g),
	\mathcal{B}^s(g)), \quad F_{\tilde{\mathscr{M}}_{\mathbb{T}}^{(m)}}(
	\mathcal{B}^s(g), \mathcal{A}^s(g)), \\
	& F_{\tilde{\mathscr{M}}_{\mathbb{T}}^{(m)}}(\mathcal{B}^s(g)), \quad
	F_{\tilde{\mathscr{M}}_{\mathbb{T}}^{(m)}}(
	\mathcal{A}^s(g), \mathcal{A}^s(g)).
\end{align*}
We calculate:
\begin{align*}
	F_{\tilde{\mathscr{M}}_{\mathbb{T}}^{(m)}}(\mathcal{A}^s(g)) & =
	F_{\tilde{\mathscr{M}}_{\mathbb{T}}^{(m)}}(s\omega - s(1-s)\sigma
	\omega^2) \\
	& = s[\tilde{Q}_m,
	\tilde{\mathbf{c}}_{\T, m}(\omega)] - s\tilde{\mathbf{c}}_{\T, m}(d
	\omega) -
	s(1-s)\tilde{\mathbf{c}}_{\T, m}(\omega^2) \\
	& =  s[\tilde{Q}_m,
	\tilde{\mathbf{c}}_{\T, m}(\omega)] + s\tilde{\mathbf{c}}_{\T,
		m}(\omega^2) -
	s(1-s)\tilde{\mathbf{c}}_{\T, m}(\omega^2) \\
	& = s[\tilde{Q}_m,
	\tilde{\mathbf{c}}_{\T, m}(\omega)] +
	s^2\tilde{\mathbf{c}}_{\T, m}(\omega^2)
\end{align*}
where we used the Maurer-Cartan equation.
Note that
\begin{equation*}
	F_{\tilde{\mathscr{M}}_{\mathbb{T}}^{(m)}}(\mathcal{A}^s(g),
	\mathcal{B}^s(g)) = (-1)^{\lvert \omega \rvert} (\tilde{\mathbf{c}}_{\T,
		m}(s\omega
	\cdot 0) - \tilde{\mathbf{c}}_{\T, m}(s \omega) \tilde{\mathbf{c}}_{\T,
		m}(0)) = 0
\end{equation*}
and the same is true for $F(\mathcal{B}^s(g), \mathcal{A}^s(g))$.
Furthermore,
\begin{align*}
	F_{\tilde{\mathscr{M}}_{\mathbb{T}}^{(m)}}(\mathcal{A}^s(g),
	\mathcal{A}^s(g)) & = (-1)^{\lvert \omega
		\rvert}(\tilde{\mathbf{c}}_{\T, m}(s^2 \omega^2) - s^2
	\tilde{\mathbf{c}}_{\T, m}(\omega)
	\tilde{\mathbf{c}}_{\T, m}(\omega)) \\
	& = -(s^2\tilde{\mathbf{c}}_{\T, m}(\omega^2) - s^2
	\tilde{\mathbf{c}}_{\T, m}(\omega)^2)
\end{align*}
and finally,
\begin{equation*}
	F_{\tilde{\mathscr{M}}_{\mathbb{T}}^{(m)}}(\mathcal{B}^s(g)) =
	F_{\tilde{\mathscr{M}}_{\mathbb{T}}^{(m)}}(\sigma \omega) =
	\tilde{\mathbf{c}}_{\T, m}(\omega).
\end{equation*}
Define $\tilde{X}_s \coloneqq s[\tilde{Q}_m, \tilde{\mathbf{c}}_{\mathbb{T},
	m}(\omega)] + s^2\tilde{\mathbf{c}}_{\mathbb{T}, m}(\omega)^2$.
Using the above results, we notice that
\begin{align*}
	F_{\tilde{\mathscr{M}}_{\mathbb{T}}^{(m)}}(\mathcal{A}^s(g),
	\mathcal{A}^s(g)) +
	F_{\tilde{\mathscr{M}}_{\mathbb{T}}^{(m)}}(\mathcal{A}^s(g)) & = s^2
	\tilde{\mathbf{c}}_{\mathbb{T}, m}(\omega)^2 -
	s^2\tilde{\mathbf{c}}_{\mathbb{T}, m}(\omega^2) \\
	& \qquad + s[\tilde{Q}_m, \tilde{\mathbf{c}}_{\mathbb{T}, m}(\omega)] +
	s^2\tilde{\mathbf{c}}_{\mathbb{T}, m}(\omega^2) \\
	& = s[\tilde{Q}_m,
	\tilde{\mathbf{c}}_{\mathbb{T}, m}(\omega)] + s^2
	\tilde{\mathbf{c}}_{\mathbb{T}, m}(\omega)^2 =
	\tilde{X}_s.
\end{align*}
\begin{lem}\label{Lemma perturbation quantization}
	For $T > 0$ and $s \in I$ the following identity holds:
	\begin{equation*}
		\sum_{M = 0}^{\infty} (-T)^M \{\underbrace{\tilde{X}_s, \ldots,
			\tilde{X}_s}_M\}_{T\tilde{Q}_m^2} = \sum_{N=0}^{\infty}
		\Phi_T^{\tilde{\mathscr{M}}_{\mathbb{T}}^{(m)}}
		(\underbrace{\mathcal{A}^s(g),
			\ldots, \mathcal{A}^s(g)}_N).
	\end{equation*}
\end{lem}
\begin{proof}
	It follows from multilinearity of the bracket that
	$(-T)^M\{\tilde{X}_s, \ldots,
	\tilde{X}_s\}_{T\tilde{Q}_m^2}$ is the sum -- over all $M \leq N \leq 2M$
	and all partitions $J \in \mathcal{P}_{M, N}$ whose subsets have one or two
	elements -- of $(-T)^M\{F_{\tilde{\mathscr{M}}_{\mathbb{T}}^{(m)}}(\theta_{J_1}),
	\ldots, F_{\tilde{\mathscr{M}}_{\mathbb{T}}^{(m)}}(\theta_{J_M})\}
	_{T\tilde{Q}_m^2}$ where $\theta_1 = \ldots = \theta_N = \mathcal{A}^s(g)$.
	On the other hand, since $F_{\tilde{\mathscr{M}}_{\mathbb{T}}^{(m)}}$ vanishes in
	arities $\geq 3$, only partitions whose subsets consist of one or two
	elements contribute to
	\begin{equation*}
		\Phi_T^{\tilde{\mathscr{M}}_{\mathbb{T}}^{(m)}}(\mathcal{A}^s(g),
		\ldots, \mathcal{A}^s(g)) = \sum_{M=0}^N (-T)^M \sum_{J \in
			\mathcal{P}_{M,N}}
			\{F_{\tilde{\mathscr{M}}_{\mathbb{T}}^{(m)}}(\theta_{J_1}),
			\ldots, F_{\tilde{\mathscr{M}}_{\mathbb{T}}^{(m)}}(\theta_{J_M})\}
			_{T\tilde{Q}^2}.
	\end{equation*}
	Summing over $M$ (respectively $N$) one sees that both sides of the
	identity in question are the same double series.
	Since the series converge absolutely in trace norm by \cite[Lemma
	4.2]{GüneysuLudewig} and
	\cref{Fundamental estimate}, their rearrangement is justified thereby
	proving the equation.
\end{proof}
\subsection{Relating the spectral flow to the pairing}\label{Section:
	spectral flow pairing 2}
We are now in a position to prove the main result of this section, which is
\cref{Thm C} from the introduction.
Recall from \cref{Section: Getzler} that
for an odd module we set $\mathrm{Ch}_{\mathscr{M}_\T} \coloneqq
\mathrm{Ch}_{\tilde{\mathscr{M}}_\T}$.
\begin{thm}\label{Thm index}
	Let $\mathscr{M} = (\HH, \mathbf{c}, Q)$ be an odd $\vartheta$-summable
	Fredholm module over a locally convex dg algebra $\Omega$ and $g \in
	\mathrm{GL}_m(\Omega^0)$.
	Then
	\begin{equation}\label{Index formula}
		\bigl\langle \mathrm{Ch}_{\mathscr{M}_\T}, \mathrm{Ch}(g)\bigr\rangle
		= \int_0^1 \Tr_{\HH^m}\bigl(\dot{Q}_{g,s} e^{-Q_{g,s}^2}\bigr) \,
		\dint s.
	\end{equation}
	In particular, if $\mathbf{c}_{\T,m}(g)$ is unitary, then
	\begin{equation*}
		\bigl\langle \mathrm{Ch}_{\mathscr{M}_\T}, \mathrm{Ch}(g)\bigr\rangle
		= \sqrt{\pi}\,\mathrm{sf}(Q_m, \mathbf{c}_{\T,m}(g)^{-1} Q_m
		\mathbf{c}_{\T,m}(g)).
	\end{equation*}
	The left-hand side of \cref{Index formula} depends only on the class of
	$\mathrm{Ch}_{\mathscr{M}_\T}$ in the Chen normalized complex and the
	class of $\mathrm{Ch}(g)$ in $\mathrm{N}^{\T,\epsilon}(\Omega)$; in
	particular, it is invariant under smooth homotopies of $g$, and under
	homotopies of $\mathscr{M}$ for which the associated
	$\mathcal{C}_1$-modules form a homotopy in the sense of
	\cref{Definition homotopy}.
\end{thm}
The last part follows from \cref{Thm ChM,Thm Chg} (applied with
$q=1$ to $\tilde{\mathscr{M}}$), and the second one from the first by
\cref{Section: spectral flow}.
For the proof of \cref{Index formula} we need
the following variant of \cref{Clifford trace cyclic sum property}, which
follows from its proof.
\begin{lem}\label{Lemma gamma cyclic}
	For operators $A_1, \ldots, A_{N-1}$ and $B$ on $\tilde{\HH}^m$ satisfying
	the assumptions from \cite[Lemma 4.2]{GüneysuLudewig} and commuting with
	$\tilde{\gamma}$, we have
	\begin{equation*}
		\sum_{k=1}^N \Tr_{\tilde{\HH}^m}(\tilde{\gamma} \{A_k, \ldots,
		A_{N-1},
		B, A_1,
		\ldots, A_{k-1}\}_{\tilde{Q}_m^2}) =
		\Tr_{\tilde{\HH}^m}(\tilde{\gamma} B \{A_1, \ldots,
		A_{N-1}\}_{\tilde{Q}_m^2}).
	\end{equation*}
\end{lem}
\begin{proof}[Proof of \cref{Index formula}]
	First of all, recall the first part of \cref{Identifications of modules and
	pairing}, the isomorphism of modules $U$ and \cref{Pairing}.
	We start with relating the operators on $\tilde{\HH}^m$ to those on
	$\HH^m$:
	Since $\omega$ is odd, $\widetilde{\mathbf{c}_{\mathbb{T}, m}}(\omega)$
	and $\widetilde{Q_m}$ are off-diagonal on
	$\widetilde{\HH^m}$ with both off-diagonal entries given by
	$\mathbf{c}_{\mathbb{T}, m}(\omega) = (\mathbf{c}(\omega_i^j))_{i,j}$ and
	$Q_m = \mathrm{diag}(Q, \ldots, Q)$, respectively.
	Both operators being odd, their
	supercommutator equals the anticommutator, so that
	$[\widetilde{Q_m}, \widetilde{\mathbf{c}_{\mathbb{T}, m}}(\omega)] =
	\mathrm{diag}\bigl([[Q_m, \mathbf{c}_{\mathbb{T}, m}(\omega)]], [[Q_m,
	\mathbf{c}_{\mathbb{T}, m}(\omega)]]\bigr)$.
	Using $\tilde{X}_sU = U\,\mathrm{diag}(X_s, X_s)$ and
	\begin{equation*}
		U^{-1} \tilde{\gamma} U =
		\begin{pmatrix}
			0 & \mathbf{1}_m \\
			\mathbf{1}_m & 0
		\end{pmatrix},
	\end{equation*}
	for all $M \in \mathbb{N}_0$ and $\tau \in \Delta_M$ we have
	\begin{multline*}
		\frac 12 \Tr_{\tilde{\HH}^m}\Bigl(\tilde{\gamma}
		\tilde{\mathbf{c}}_{\mathbb{T}, m}(\omega)\prod_{i=1}^{M}
		\bigl(e^{-(\tau_i - \tau_{i-1})\tilde{Q}_{m}^2}\tilde{X}_{s}\bigr)
		e^{-(1-\tau_{M})\tilde{Q}_{m}^2} \Bigr) \\
		= \Tr_{\HH^m}\Bigl(\mathbf{c}_{\mathbb{T}, m}(\omega)\prod_{i=1}^{M}
		\bigl(e^{-(\tau_i - \tau_{i-1})Q_{m}^2}X_{s}\bigr)
		e^{-(1-\tau_{M}) Q_{m}^2} \Bigr).
	\end{multline*}
	Remember that $\dot{Q}_{g,s} = \mathbf{c}_{\T,m}(\omega)$.
	Now, combining \cref{Perturbation series} with the above identity
	integrated over $\Delta_M$ and invoking \cref{Lemma perturbation
	quantization} with $T = 1$ yields
	\begin{align*}
		\int_0^1 \Tr_{\HH^m}(\dot{Q}_{g,s} e^{-Q_{g,s}^2}) \, \dint s
		& = \int_0^1 \frac 12 \Tr_{\tilde{\HH}^m}\Bigl(\tilde{\gamma}
		\tilde{\mathbf{c}}_{\mathbb{T}, m}(\omega) \sum_{M = 0}^{\infty}(-1)^M
		\{\tilde{X}_s, \ldots, \tilde{X}_s\}_{\tilde{Q}^2_m} \Bigr)\,
		\dint s \\
		& = \sum_{N=0}^{\infty} \int_0^1 \frac 12
		\Tr_{\tilde{\HH}^m}\Bigl(\tilde{\gamma}\tilde{\mathbf{c}}_{\mathbb{T},
			m}(\omega)
		\Phi_1^{\tilde{\mathscr{M}}_{\mathbb{T}}^{(m)}}(\mathcal{A}^s(g),
		\ldots, \mathcal{A}^s(g)) \Bigr)\, \dint s,
	\end{align*}
	where we used Fubini's theorem in the last step which is justified by
	\cref{Fundamental estimate}.
	By construction, the operators appearing in the brackets of
	\begin{equation*}
		\Phi_1^{\tilde{\mathscr{M}}_{\mathbb{T}}^{(m)}}\left(
		\mathcal{A}^s(g)^{\otimes(k-1)} \otimes
		\mathcal{B}^s(g)
		\otimes \mathcal{A}^s(g)^{\otimes(N-k)}
		\right)
	\end{equation*}
	supercommute with the $\mathcal{C}_1$-action on $\tilde{\HH}^m$ and
	therefore commute with $\tilde{\gamma}$.
	Moreover, because $F(\mathcal{A}^s(g), \mathcal{B}^s(g)) = 0 =
	F(\mathcal{B}^s(g), \mathcal{A}^s(g))$, and $F(\mathcal{B}^s(g)) =
	\tilde{\mathbf{c}}_{\mathbb{T}, m}(\omega)$, the argument of
	\cref{Ch pullback}, replacing \cref{Clifford
	trace cyclic sum property} by \cref{Lemma gamma cyclic}, shows that
	\begin{multline*}
		\Tr_{\tilde{\HH}^m}\left(\tilde{\gamma}
		\sum_{k=1}^{N}\Phi_1^{\tilde{\mathscr{M}}_{\mathbb{T}}^{(m)}}\left(
		\mathcal{A}^s(g)^{\otimes(k-1)} \otimes
		\mathcal{B}^s(g)
		\otimes \mathcal{A}^s(g)^{\otimes(N-k)}
		\right)\right) \\ = -\Tr_{\tilde{\HH}^m}\left(\tilde{\gamma}
		\tilde{\mathbf{c}}_{\mathbb{T}, m}(\omega)
		\Phi_1^{\tilde{\mathscr{M}}_{\mathbb{T}}^{(m)}}
		(\mathcal{A}^s(g)^{\otimes(N-1)})\right)
	\end{multline*}
	for $N \geq 1$.
	Substituting this into previous computation and shifting the summation
	index, we have
	\begin{multline*}
		\int_0^1 \Tr_{\HH^m}(\dot{Q}_{g,s} e^{-Q_{g,s}^2}) \, \dint s
		\\ = -\sum_{N=1}^{\infty} \int_0^1 \frac 12
		\Tr_{\tilde{\HH}^m}\left(\tilde{\gamma}
		\sum_{k=1}^{N}\Phi_1^{\tilde{\mathscr{M}}_{\mathbb{T}}^{(m)}}\left(
		\mathcal{A}^s(g)^{\otimes(k-1)} \otimes
		\mathcal{B}^s(g)
		\otimes \mathcal{A}^s(g)^{\otimes(N-k)}
		\right)\right)\, \dint s,
	\end{multline*}
	which is $\langle
	\mathrm{Ch}_{\mathscr{M}_\T}, \mathrm{Ch}(g)\rangle$ by
	\cref{Pairing} since $\mathrm{Ch}_0(g) = 0$.
\end{proof}
\section{Example: odd-dimensional spin manifolds}\label{Section: example}
Let $X$ be a closed Riemannian spin manifold, $\Sigma
\to X$ its spinor bundle and $\mathrm{D}$ the associated Dirac operator, and
consider the Hilbert space $\HH \coloneqq \Gamma_{L^2}(X, \Sigma)$ of square
integrable sections.
Define $\mathbf{c} \colon \Omega(X) \to
\Lop(\HH)$ by letting $\mathbf{c}(f)$ be multiplication by $f$ for $f \in
\Omega^0(X) = C^\infty(X)$ and extend this to differential forms via
\begin{equation*}
	\mathbf{c}(f\,\theta_1 \wedge \ldots \wedge \theta_k) \coloneqq
	\frac{1}{k!}\sum_{\pi \in S_k} \mathrm{sgn}(\pi)\, f\,
	c(\theta_{\pi(1)}) \cdots c(\theta_{\pi(k)})
\end{equation*}
for $\theta_1, \ldots, \theta_k \in \Omega^1(X)$, where $c(\theta)$ denotes
Clifford multiplication by $\theta$.
Since the principal symbol of the Dirac operator is Clifford multiplication, we
have $[\mathrm{D}, f] = \mathbf{c}(df)$ and
$\mathbf{c}$ is $C^\infty(X)$-linear on both sides, so that the relations
\eqref{Equations Fredholm} hold.
The analytic conditions (A1) and (A2) follow from classical elliptic estimates
as well as heat kernel estimates for $\mathrm{D}^2$, see \cite[Example
2.3]{GüneysuLudewig}.
Hence, when $\dim X$ is odd (and $\HH$ is regarded as ungraded),
\begin{equation*}
	\mathscr{M}_X \coloneqq (\Gamma_{L^2}(X, \Sigma), \mathbf{c}, \mathrm{D})
\end{equation*}
is an odd $\vartheta$-summable Fredholm module over the locally convex dg
algebra $\Omega(X)$ in the sense of \cref{Definition odd FM}.

Now let $g \colon X \to \mathrm{U}(m)$ be smooth, so that $g \in
\mathrm{GL}_m(\Omega^0(X))$ and $\mathbf{c}_{\T, m}(g)$ is the unitary operator of
multiplication by $g$ on $\HH^m = \Gamma_{L^2}(X, \Sigma \otimes \C^m)$.
\cref{Thm index} yields:
\begin{cor}\label{Cor spin}
	For every smooth map $g \colon X \to \mathrm{U}(m)$,
	\begin{equation*}
		\langle \mathrm{Ch}_{(\mathscr{M}_X)_\T}, \mathrm{Ch}(g)
		\rangle = \int_0^1 \Tr(\dot{\mathrm{D}}_s
		e^{-\mathrm{D}_s^2})\,
		\dint s = \sqrt{\pi}\,\mathrm{sf}(\mathrm{D}_m,
		\mathbf{c}_{\T, m}(g)^{-1} \mathrm{D}_m \mathbf{c}_{\T, m}(g)),
	\end{equation*}
	where $\mathrm{D}_m = \mathrm{D} \otimes \mathbf{1}_{\C^m}$ and
	$\mathrm{D}_s = (1-s)\mathrm{D}_m + s\mathbf{c}_{\T, m}(g)^{-1}
	\mathrm{D}_m \mathbf{c}_{\T, m}(g)$.
\end{cor}
The odd index theorem implies that the spectral flow on the right-hand
side equals, up to normalization, $\int_X \hat{A}(X) \wedge \mathrm{ch}^-(g)$,
see
\cite[Eq. (1)]{GüneysuCacciatori} and the references given there.
Moreover, Güneysu-Cacciatori show that $\mathrm{Ch}(g)$ is mapped to the odd
Bismut-Chern character of $g$ on the free loop space of $X$ under the Chen iterated integral map.
\cref{Cor spin} therefore provides the
cochain side of the loop space picture of the odd index theorem: it realizes
the spectral flow as the pairing of an entire cyclic cocycle, built from the
Dirac operator, with the chain representing the odd Bismut-Chern character.
\bibliographystyle{amsalpha}
\bibliography{ref}

\providecommand{\bysame}{\leavevmode\hbox to3em{\hrulefill}\thinspace}
\providecommand{\MR}{\relax\ifhmode\unskip\space\fi MR }
\providecommand{\MRhref}[2]{%
  \href{http://www.ams.org/mathscinet-getitem?mr=#1}{#2}
}
\providecommand{\href}[2]{#2}
\begin{thebibliography}{GJP91}

\bibitem[Ati85]{AtiyahCircular}
Michael~F. Atiyah, \emph{Circular symmetry and stationary-phase approximation},
  Colloque en l'honneur de {L}aurent {S}chwartz, {V}ol.~1 ({P}alaiseau, 1983)
  (Michel~R. Herman and Bernard Teissier, eds.), Ast{\'e}risque, vol. 131,
  Soci{\'e}t{\'e} Math{\'e}matique de France, 1985, pp.~43--59.

\bibitem[Bis85]{BismutCMP85}
Jean-Michel Bismut, \emph{Index theorem and equivariant cohomology on the loop
  space}, Communications in Mathematical Physics \textbf{98} (1985), no.~2,
  213--237.

\bibitem[CG21]{GüneysuCacciatori}
Sergio~L. Cacciatori and Batu G{\"u}neysu, \emph{Odd characteristic classes in
  entire cyclic homology and equivariant loop space homology}, Journal of
  Noncommutative Geometry \textbf{15} (2021), no.~2, 615--642.

\bibitem[Con88]{ConnesCyclicCohomology}
Alain Connes, \emph{Entire cyclic cohomology of {B}anach algebras and
  characters of $\theta$-summable {F}redholm modules}, K-Theory \textbf{1}
  (1988), no.~6, 519--548.

\bibitem[Con89]{ConnesCompact}
\bysame, \emph{Compact metric spaces, {F}redholm modules, and hyperfiniteness},
  Ergodic Theory and Dynamical Systems \textbf{9} (1989), no.~2, 207--220.

\bibitem[Get93]{GetzlerOdd}
Ezra Getzler, \emph{The odd {C}hern character in cyclic homology and spectral
  flow}, Topology \textbf{32} (1993), no.~3, 489--507.

\bibitem[GJP91]{GetzlerJonesPetrack}
Ezra Getzler, John D.~S. Jones, and Scott Petrack, \emph{Differential forms on
  loop spaces and the cyclic bar complex}, Topology \textbf{30} (1991), no.~3,
  339--371.

\bibitem[GL22]{GüneysuLudewig}
Batu G{\"u}neysu and Matthias Ludewig, \emph{The {C}hern character of
  $\vartheta$-summable {F}redholm modules over dg algebras and localization on
  loop space}, Advances in Mathematics \textbf{395} (2022), 108143.

\bibitem[GM24]{GüneysuMiehe}
Batu G{\"u}neysu and Jonas Miehe, \emph{Fermionic {D}yson expansions and
  stochastic {D}uistermaat--{H}eckman localization on loop spaces}, 2024,
  Preprint, \href{https://arxiv.org/abs/2410.14034}{arXiv:2410.14034}.

\bibitem[GS89]{GetzlerSzenes}
Ezra Getzler and Andr{\'a}s Szenes, \emph{On the {C}hern character of a
  theta-summable {F}redholm module}, Journal of Functional Analysis \textbf{84}
  (1989), no.~2, 343--357.

\bibitem[JLO88]{JLO}
Arthur Jaffe, Andrzej Lesniewski, and Konrad Osterwalder, \emph{Quantum
  {$K$}-theory. {I}. {T}he {C}hern character}, Communications in Mathematical
  Physics \textbf{118} (1988), no.~1, 1--14.

\bibitem[Kas80]{Kasparov}
G.~G. Kasparov, \emph{The operator {$K$}-functor and extensions of
  {$C^*$}-algebras}, Izv. Akad. Nauk SSSR Ser. Mat. \textbf{44} (1980), no.~3,
  571--636.

\bibitem[Kat95]{Kato}
Tosio Kato, \emph{Perturbation theory for linear operators}, second ed.,
  Classics in Mathematics, no. 132, Springer-Verlag, Berlin, 1995, Corrected
  printing.

\bibitem[Mie24]{MieheThesis}
Jonas Miehe, \emph{The {C}hern character of $\vartheta$-summable
  $\mathcal{C}_q$-{F}redholm modules}, Ph.D. thesis, Technische Universit{\"a}t
  Chemnitz, 2024,
  \href{https://nbn-resolving.org/urn:nbn:de:bsz:ch1-qucosa2-910433}{urn:nbn:de:bsz:ch1-qucosa2-910433}.

\bibitem[Qui85]{Quillen}
Daniel Quillen, \emph{Superconnections and the {C}hern character}, Topology
  \textbf{24} (1985), no.~1, 89--95.

\bibitem[SS10]{Simons}
James Simons and Dennis Sullivan, \emph{Structured vector bundles define
  differential {$K$}-theory}, Quanta of Maths (Etienne Blanchard, David
  Ellwood, Masoud Khalkhali, Matilde Marcolli, Henri Moscovici, and Sorin Popa,
  eds.), Clay Mathematics Proceedings, vol.~11, American Mathematical Society
  and Clay Mathematics Institute, 2010, pp.~579--599.

\end{thebibliography}
\end{document}